\documentclass{article}
\usepackage[utf8]{inputenc}
\usepackage{amsmath}
\usepackage{amsthm}
\usepackage{amssymb}
\usepackage{graphicx}
\usepackage{bbm}
\usepackage{yfonts}
\usepackage{scalerel,stackengine}
\usepackage{xcolor}
\usepackage[letterpaper,top=2cm,bottom=2cm,left=2cm,right=2cm,marginparwidth=1.75cm]{geometry}
\usepackage[colorlinks=true, allcolors=blue]{hyperref}
\usepackage{comment}
\usepackage{natbib}
\usepackage{authblk}

\newcommand{\ind}{\mathbf{1}}

\newcommand{\E}{ \mathbb{E}}
\newcommand{\Pro}{\mathcal{P}}

\newcommand{\cQ}{\mathcal{Q}}

\newcommand{\gm}{\mathfrak{m}}
\newcommand{\gn}{\mathfrak{n}}

\newcommand{\Diam}{\textrm{Diam }}

\newcommand{\bert}[1]{\noindent \textcolor{blue}{\textsf{[B: #1]}}}

\newtheorem{theorem}{Theorem}
\newtheorem{remark}{Remark}
\newtheorem{definition}{Definition}
\newtheorem{lemma}{Lemma}
\newtheorem{corollary}{Corollary}
\newtheorem{proposition}{Proposition}

\numberwithin{equation}{section}
\title{Concentration of the empirical measure in Wasserstein distance: bounds involving the covering dimension}
 
\author[1]{Jérôme Dedecker}
\affil[1]{Université Paris Cité, UMR CNRS  8145, Laboratoire MAP5, 75006, Paris, France.}
\author[2]{Aurélie Fischer}
\affil[2]{Université Paris Cité, UMR CNRS 8001, Laboratoire de Probabilités, Statistique et Modélisation, 75013, Paris, France}
\author[3]{Bertrand Michel}
\affil[3]{Centrale Nantes, Nantes Universit\'e, Laboratoire de Math\'ematiques Jean Leray UMR CNRS 6629,
France}

\date{}

\begin{document}

\maketitle

\begin{abstract}
 We give  concentration inequalities in Wasserstein distance for the empirical measure of a sequence of independent and identically distributed random  variables with values in a Polish  space $E$. These inequalities involve the covering dimension of the support of the distribution of the variables. More precisely, we obtain a complete extension of the concentration inequalities of \cite{fournier2015rate} in the case where $E={\mathbb R}^d$, in which the covering dimension replaces the dimension of the ambient space $E$.   
\end{abstract}


\smallskip


\section{Introduction}
Let $(\Omega, {\mathcal A}, {\mathbb P})$ be a probability space,  and $(X_i)_{1 \leq i \leq n}$ be a sequence of independent and identically distributed (i.i.d.) random variables with values in a Polish space $(E,d)$. Let $\mu$ be the common distribution of the variables $X_i$, and 
$$L_n = \frac 1 n \sum_{i=1}^n \delta_{X_i}$$ be the empirical measure associated with the sample $(X_i)_{1 \leq i \leq n}$. 
A natural question in geometric inference and topological data analysis  is to determine precisely how the empirical measure $L_n$ approximates the measure $\mu$ when the data is distributed over a subset $S$ of $E$, depending on the topological properties of $S$  (which is not necessarily a manifold).  

Since the measures $L_n$ and $\mu$ are singular in most cases, a family of natural distances that can be used to evaluate the difference between these measures is that of Wasserstein distances
$W_p$, $p \geq 1$. The Wasserstein distance $W_p$ is defined on the set $\Pro_p(E)^2$ of couples of measures with a finite $p$-th moment by
\begin{equation*}
 W_p^p (\mu, \nu) = \inf_{ \pi \in \Pro(\mu, \nu)} \int d^p(x, y) \pi(dx, dy)
\end{equation*}
where the infimum is taken on the set $\Pro(\mu, \nu)$ of probability measures with  marginals $\mu$ and $\nu$.  Then $W_p$ defines a metric on $\Pro_p(E)$. Recall that $W_1$ is better known as the Kantorovich distance, and has the simple dual form:
$$
W_1(\mu, \nu) = \sup_{f \in \text{Lip}_1(E)} \left | \mu(f)- \nu (f) \right | \, ,
$$
where Lip$_1(E)$ is the set of functions from $E$ to ${\mathbb R}$ such that $|f(x)-f(y)| \leq d(x,y)$. We refer to \cite{rachev1998mass,villani2008optimal} for an  extensive treatment of these distances.

Let us return to our initial problem. Let $S$ be a subset of $E$ such that $\mu(S)=1$. We will call this set the {\em support} of $\mu$, but it may well be a larger set than the usual support,  on which we know the data is distributed. Let us also define the covering numbers of a bounded  set $X \subset E$. 
\begin{definition}
For a bounded set $X \subset E$, the covering number of order $\delta$, denoted by $N(X, \delta)$, is defined as
the minimal $n \in \mathbb{N}$ such that there exist $x_1, \ldots, x_n$ in $X$ with
\begin{equation*}
 X \subset \bigcup_{j = 1}^n B(x_i, \delta) \, 
\end{equation*}
where $B(x,r)$ denotes the closed ball of radius $r$ centered at $x$.
\end{definition}
An initial answer to our question is provided by \cite{boissard2014mean} when the set $S$ is totally bounded, with diameter $\Delta$. More precisely, if the set $S$ is of finite-dimensional type, in the sense that there exist $\alpha, \beta >0$ such that for any $\delta \in (0,\Delta] $,
\begin{equation}
\label{eq:dimmetric}
    N(S, \delta) \leq  \beta \left( \frac \Delta \delta \right) ^\alpha ,
\end{equation} then, if $\alpha >2p$, 
$$
{\mathbb E}(W_p(L_n, \mu)) \leq C_{\alpha, \beta, p} \frac{\Delta}{n^{1/\alpha}} \, .
$$
where the positive constant $C_{\alpha, \beta, p}$ depends only on $\alpha, \beta, p$. As indicated by \cite{boissard2014mean}, the quantity $\alpha$ in \eqref{eq:dimmetric} plays the role of a dimension. With a slight abuse of language, we will say that  $\alpha$ is the {\em covering dimension} of $S$. 
The case of unbounded support is also considered by Boissard and Le Gouic. In their Corollary 1.3, they provide a control of ${\mathbb E}(W_p(L_n, \mu))$ under a moment assumption on $\mu$, and assuming that $N(S\cap B(x_0, r), \delta)$ satisfies the condition \eqref{eq:dimmetric} uniformly in $r>0$ (see the condition \eqref{eq:dimmetric2} below).

In this article, we adopt the general framework considered by \cite{boissard2014mean}, with the same kind of conditions on the covering dimension of $S$. However, a more appropriate notion of dimension for Wasserstein distances has been introduced by \cite{Weed19}. This Wasserstein dimension - according to the author's terminology - is no longer simply related to the support of $\mu$, but also to the measure $\mu$ itself. The relevance of this notion of dimension is clearly highlighted in Theorem 1 of \cite{Weed19}. Note also that, in the same article, the authors give a sub-Gaussian concentration inequality for $W_p^p(L_n, \mu)$ around its expectation (McDiarmid-type inequality), in the case where the support $S$ is bounded. 

Let us now turn to the known results when $E={\mathbb R}^d$, without taking into account the intrinsic dimension of the support of $\mu$. The most comprehensive article on this subject is that of \cite{fournier2015rate}, in which the authors give several concentration inequalities in Wasserstein distance, i.e. inequalities allowing to control the deviation ${\mathbb P}\left (W_p^p(L_n, \mu)>x \right )$. In particular, the authors give Hoeffding-type concentration inequalities in the case where $S$ is bounded (see their Proposition 10), Bernstein-type inequalities in the case where $\mu$ has a sub/super exponential  moment (see their Theorem 2 cases (1) and (2)), and Fuk-Nagaev-type inequalities when $\mu$ has a strong moment of order $q>2p$ (see their Theorem 2 case (3)). This last inequality was improved in the article by \cite{BehaviorDedeckerMerl19}, who obtain Fuk-Nagaev-type inequalities under weak moment assumptions (see their Theorem 2.1 in case of weak moments of order $q \in (p,2p)$ and Theorem 2.3 in case of weak moments of order $q>2p$).

In the present article, we show that all the concentration inequalities obtained by \cite{fournier2015rate} and \cite{BehaviorDedeckerMerl19} remain true when replacing the ambient dimension $d$ with the covering dimension $\alpha$ of the support of $\mu$. Furthermore, in the case where $\mu$ has a weak moment of order $q \in (p, 2p)$, our results are more precise than those obtained in Theorem 2.1 of \cite{BehaviorDedeckerMerl19}.

To obtain these concentration results, we will combine several ingredients: in the bounded case, we apply the dyadic transport Lemma as written in Proposition 1 of \cite{Weed19} (see also \cite{DSS13,boissard2014mean,fournier2015rate} for similar results), then Rosenthal's inequality for the norm $r\geq 2$,  with the optimal constants depending on  $r$, as described by \cite{pinelis1994optimum}. In the unbounded case, we start from the decomposition described in Section 2 of \cite{boissard2014mean}, then we adapt some calculations from \cite{fournier2015rate} and \cite{BehaviorDedeckerMerl19}. We wish to emphasise the decisive importance of the optimal Rosenthal inequality in our proof strategy: it is thanks to this version of Rosenthal's inequality that we obtain our first estimates for $\|W_p^p(L_n, \mu)\|_r$ for $r\geq 2$ (see our Proposition \ref{prop:bound-r-bounded}), with the correct dependence on $r$ and $\alpha$. It is then sufficient to optimise in $r\geq 2$ to obtain Hoeffding-type inequalities (see our Proposition \ref{PropHoeff}). These estimates on $\|W_p^p(L_n, \mu)\|_r$ are reused in the unbounded case to control  the main term in the decomposition \eqref{borne3termes-bis}, leading to the key Lemma \ref{expobound-mainterm}.

Before stating and demonstrating our results, let us conclude this introduction with a result on large deviations, which is interesting in itself. Let $x_0 \in E$, and let
\begin{equation}\label{DefTail}
  H(t)=  \mu(B(x_0,t)^c)   
\end{equation} be the tail distribution of the measure $\mu$. 
For any $q \geq 1$, the weak moment of order $q$ of a random variable $X$ with distribution  $\mu$  (or equivalently the weak moment of order $q$ of $\mu$) is then defined by 
\begin{equation}\label{DefWeakMoment}
\|X\|_{q, w}^q:= \sup_{t \geq 0} t^q H(t) \, .
\end{equation}
Our large deviations result is the following (it is a particular case of Corollaries \ref{modev_unbounded-alphageq2p}  and \ref{modev_unbounded-alphaleq2p} of Section \ref{SecUnbounded}, by taking $\rho=1$): if the support $S$ is of finite dimensional-type with covering dimension $\alpha$ (meaning that there exist $\alpha>0, \beta >0$ such that inequality \eqref{eq:dimmetric2} of Section \ref{SecUnbounded} holds), and if $\mu$ has a weak moment of order $rp$ for some $r>1$ (except for $r=2$, and $r= (\alpha-p)/\alpha$ in case $ \alpha>2p$, for which the moment condition needs to be slightly reinforced), then 
\begin{equation}\label{LDbound}
\limsup_{n \rightarrow \infty} n^{r-1} {\mathbb P}\left (W_p^p(L_n, \mu)>x \right ) \leq C_{p,\alpha,\beta,r}   \frac{\|X\|^{rp}_{rp, w}}{x^r} \, ,
\end{equation}
where the positive constant $C_{p,\alpha,\beta,r}$ depends only on 
${p,\alpha,\beta,r}$. The result \eqref{LDbound} is remarkable, because it shows that the rate of decay of ${\mathbb P}\left (W_p^p(L_n, \mu)>x \right )$ is always (at worst)   $O(1/n^{r-1})$ regardless of the structure of the ambient space $E$, provided that $S$ is of finite-dimensional type, with the intrinsic dimension $\alpha$ only appearing in the constant $C_{p,\alpha,\beta,r}$.

As for the numerical constants, which depend among other parameters on $\alpha$ and $p$, we could display explicit constants by being more precise at certain steps in the proofs, but this would be pointless and unnecessarily complicated, as we have made no effort to optimise these constants. Obtaining optimal constants in these concentration inequalities is an important topic, but it may require different proof techniques. A first step in this direction has been made by \cite{Fournier23} for the quantity ${\mathbb E}(W_p^p(L_n, \mu))$, in case $E={\mathbb R}^d$ (but the constants here depend on the ambient dimension $d$).

The paper is organised as follows. In Section \ref{Sec2}, we give moment bounds and Hoeffding-type concentration inequalities for $W_p^p(L_n, \mu)$ when the support $\mu$ is totally bounded. As a consequence of the maximum versions of these inequalities, we deduce almost sure rates of convergence for $W_p^p(L_n, \mu)$. In Section \ref{SecUnbounded}, we consider the case where the support $\mu$ is unbounded, and we give Fuk-Nagev-type inequalities (under weak moment hypotheses) as well as Bernstein-type inequalities (under sub/super exponential moment hypotheses) for $W_p^p(L_n, \mu)$. Section \ref{Sec4} is devoted to proofs of the results in the unbounded case. In the appendix, for the sake of completeness, we recall the moment inequalities for partial sums of i.i.d. random variables that are used in the proofs.

In the rest of this article (as we have already done in this introduction), we denote by $C_{a_1, a_2, \ldots, ak}$ a positive constant that depends only on $a_1, a_2, \ldots, a_k$. Furthermore, the constants $C_{a_1, a_2, \ldots, ak}$ may change from  line to line; finally, two constants $C_{a_1, a_2, \ldots, ak}$ appearing in the same equation may have different values.






\section{Totally bounded support}\label{Sec2}

\subsection{Moment inequalities}

Recall that $S$ is a set such that $\mu(S)=1$ (which we call the support of $\mu$). Using a sequence of partitions of $S$ proposed by \cite{boissard2014mean}, as well as a dyadic transport lemma from the paper by \cite{Weed19}, we first give an analogue of Theorem 1.1 from \cite{boissard2014mean} that allows us to control the moments of order $r\geq 2$ of $W_p ^p ( \mu,L_n)$.
\begin{proposition}   
\label{prop:wppr_bounded}
Assume that the support $S$ of $\mu$ is  bounded with diameter $\Delta$, and that $N(S, 4^{-(k^*+2)}\Delta)< \infty$ for some integer $k^*\geq 1$. Then,  for any $ r \geq 2 $ and any $p\geq 1$, 
$$
\| W_p ^p ( \mu,L_n)\|_r  \leq
\frac{\Delta ^p}{4^{k^* p}}  +
C_p  \int_{ 4^{-k^*}\Delta /16}^{\Delta /16}
\delta^{p-1} \bigg[ \sqrt{\frac r n} \sqrt{N(S, \delta)}+ \frac{r}{n} N(S,\delta)  \bigg]  d \delta 
     $$
where $C_p$ only depends on $p$.
\end{proposition}
 
 
\begin{proof}
We assume for now that the diameter $\Delta$ of $S$ is equal to one, as in  \cite{Weed19}. By applying Lemma 2.1 in~\cite{boissard2014mean} to $S$ with $s = \frac 1 4$, since we assume that $N(S, 4^{-(k^*+2)})< \infty$, 
we can find a sequence of refined partitions $\cQ^1 , \dots, \cQ^{k^*}$ of $S$  such that
\begin{enumerate}
 \item  for each $k$, $\cQ^k = (Q_i^k)_{1 \leq i \leq m(k)}$ is a partition of $S$ where 
$
 m(k) \leq N(S, 4^{-k-1} )
$,
 \item  $\Diam Q_i^k \leq 4^ {-k}   $,
 \item  for each $1 \leq k \leq k^*$ and each $1 \leq i \leq m(k)$ there exists $1 \leq i' \leq m(k-1)$ such that $Q_i^k \subset Q_{i'}^{k-1}$.
\end{enumerate}
 
Next, we apply Proposition~1 in \cite{Weed19} for this refined sequence of partitions. We obtain the upper bound
\begin{equation*}
W_p^p(\mu, L_n) \leq \frac{1}{4^{k^* p}} + 4^p\sum_{k = 1}^{k^*} \frac{1}{4^{k p}} \sum_{Q_i^k \in \cQ^k} |\mu(Q_i^k) - L_n(Q_i^k)|\,. 
\end{equation*}
Consequently
\begin{eqnarray*}
\left\| W_p^p(\mu, L_n) \right\|_r 
&\leq& \frac{1}{4^{k^* p}} + 4^p   \sum_{k = 1}^{k^*} \frac{1}{4^{k p}}  \left\| \sum_{Q_i^k \in \cQ^k}   \mu(Q_i^k) - L_n(Q_i^k)  \right\|_r \\
&\leq & \frac{1}{4^{k^* p}} + 4^p  \sum_{k = 1}^{k^*} \frac{1}{4^{k p}}  \sum_{Q_i^k \in \cQ^k} \left\|  \mu(Q_i^k) - L_n(Q_i^k)  \right\|_r \\
&\leq & \frac{1}{4^{k^* p}} + 4^p   \sum_{k = 1}^{k^*} \frac{1}{4^{k p}}   \sum_{Q_i^k \in \cQ^k} L
\bigg[ \sqrt{\frac r n} \sqrt{\mu(Q_i^k)}+\frac{r}{n}\bigg] 
\end{eqnarray*}
by applying the Rosenthal inequality (recalled in  Theorem~\ref{Lem-Ros} of the Appendix). Next, by Cauchy-Schwarz inequality, and according to the definition of the partition $\cQ^k$, one obtains
\begin{eqnarray*}
\left\| W_p^p(\mu, L_n) \right\|_r 
& \leq&  \frac{1}{4^{k^* p}} + 4^p  L\sum_{k = 1}^{k^*} \frac{1}{4^{k p}}    
\bigg[ \sqrt{\frac r n} \sqrt{m(k)}+ m(k)  \frac{r}{n}\bigg]  \\
& \leq&  \frac{1}{4^{k^* p}} + 4^p   L\sum_{k = 1}^{k^*} \frac{1}{4^{k p}}    
\bigg[ \sqrt{\frac r n} \sqrt{N(S, 4^{-k-1} )}+ \frac{r}{n} N(S, 4^{-k-1} )  \bigg] \\
& \leq& \frac{1}{4^{k^* p}} + C_p  \int_{4^{-k^*-2}}^{4^{-2}} 
\delta^{p-1}\bigg[ \sqrt{\frac r n} \sqrt{N(S, \delta)}+ \frac{r}{n} N(S,\delta)  \bigg]  d \delta  
\end{eqnarray*}
where $C_p$ only depends on $p$. In the last line, we used the mean value inequality for the functions $t\mapsto t^{p-1}N(S,t)^{1/i}$, where $i=1,2$. Indeed, for $k=1,\dots,k^*$, for every $4^{-k-2}\leq t\leq 4^{-k-1} $, 
$$4^{-(k+2)(p-1)}   N(S,4^{-k-1})^{1/i}\leq t^{p-1}N(S,t)^{1/i},$$ and thus, $$4^{-(k+2)(p-1)}  N(S,4^{-k-1}  )^{1/i}\times 3 \times4^{-k-2} \leq \int_{4^{-k-2} }^{4^{-k-1} }t^{p-1}N(S,t)^{1/i}dt.$$
By summing in $k$, we obtain the announced result in the Proposition for $\Delta = 1$.

We now consider the  case with a general diameter $\Delta$. We introduce the metric $d' = \frac d \Delta$ for which the diameter of $S$ is one. Using obvious notation, we have $N(S, \delta,d') = N(S, \Delta \delta,d)$ and thus 
\begin{eqnarray*}
\|  W_p ^p ( \mu,L_n,d) \|_r & \leq &
\|  \Delta ^p W_p ^p ( \mu,L_n,d')\|_r \\
& \leq & \frac{\Delta ^p}{4^{k^* p}} + \Delta ^p  C_p  \int_{4^{-k^*}/16}^{1/16}
\delta^{p-1} \bigg[ \sqrt{\frac r n} \sqrt{N(S, \delta,d')}+ \frac{r}{n} N(S,\delta,d')  \bigg]  d \delta  \\
& \leq &  \frac{\Delta ^p}{4^{k^* p}} +  C_p  \int_{4^{-k^*}/16}^{1/16} 
  (\Delta\delta)^{p-1}  \bigg[ \sqrt{\frac r n} \sqrt{N(S, \Delta \delta,d)}+ \frac{r}{n} N(S, \Delta \delta,d)  \bigg] \Delta  d \delta \\
& \leq &  \frac{\Delta ^p}{4^{k^* p}} + C_p  \int_{ 4^{-k^*}\Delta/16}^{\Delta/16}
   \delta ^{p-1}  \bigg[ \sqrt{\frac r n} \sqrt{N(S,  \delta,d)}+ \frac{r}{n} N(S,  \delta,d)  \bigg]    d \delta .
\end{eqnarray*}
\end{proof}

In the rest of this section, we assume that $S$ is totally bounded, meaning that $N(S, \delta)$ is finite for any $\delta >0$. More precisely, as  in~\cite{boissard2014mean}, 
we assume that the support $S$ of $\mu$ is of finite-dimensional  type, in the sense that there exist $\alpha>0, \beta >0$ such that \eqref{eq:dimmetric} is satisfied. 

\begin{proposition}
\label{prop:bound-r-bounded}
Under Assumption~\eqref{eq:dimmetric}, for any $r \geq 2$ and $p \geq 1  $, one has
\[
\frac1 {\Delta ^p} \left\| W_p^p(\mu, L_n) \right\|_r \leq 
\begin{cases}
			C_{\alpha,p,\beta} \,  \sqrt{\frac rn}  &    \text{if } \alpha < 2p ; \\
            C_{\alpha, \beta} \,   \sqrt{\frac rn} \log (e +\frac nr)      &   \text{if } \alpha =  2p ; \\            
            C_{\alpha,p,\beta} \,  \left ( \frac rn \right ) ^{\frac{p}{\alpha}}   &   \text{if } \alpha >   2p.
		 \end{cases} 
\]
\end{proposition}
\begin{proof}

According to Proposition \ref{prop:wppr_bounded}, for any $k^* \geq 1$, one has 
\begin{equation}
\label{eq:intborne}    
\left\| W_p^p(\mu, L_n) \right\|_r
 \leq   \frac{\Delta ^p}{4^{k^* p}} +
C_p \sqrt{ \beta} \sqrt{\frac r n}  \int_{4^{-k^*}\Delta/16}^{\Delta/16 }  
  \Delta^{\alpha/2}  \delta^{p-1-\alpha/2} d \delta 
  +
C_p \beta \frac{r}{n}  \int_{4^{-k^*}\Delta/16}^{\Delta/16}  \Delta^{\alpha}  \delta^{p-1-\alpha}  d \delta .
\end{equation}
For $\alpha \notin\{ p,2 p\} $, it yields
\begin{equation}
\label{eq:intborne2}    
\left\| W_p^p(\mu, L_n) \right\|_r
 \leq   \frac{\Delta ^p}{4^{k^* p}} +
C_{p, \alpha} \sqrt{  \frac {r\beta } n} \Delta ^p  \left[ 1 -  4^{- k^*(p - \alpha /2)} \right] 
  +
C_{p, \alpha}  \frac{r\beta}{n} \Delta ^p   \left[ 1 -  4^{- k^*(p - \alpha )} \right] .
\end{equation}
Let us now consider the five cases:

\medskip

\noindent $\bullet$ If $\alpha <  p$ then the two integrals in \eqref{eq:intborne} converge, so by  letting $k^*$ tend to infinity in \eqref{eq:intborne2}, we get
\begin{equation*}
\left\| W_p^p(\mu, L_n) \right\|_r 
\leq C_{\alpha,p,\beta}    \Delta^p  \left\{  \sqrt{ \frac{r}{n}}  + 
    \frac{r }{n}
\right\} .
\end{equation*}
If $n \leq r$, the inequality of Proposition \ref{prop:bound-r-bounded} is obviously true (with $C_{\alpha,p,\beta}=1$). If $n > r$, then $\frac r n \leq \sqrt{ \frac r n}$, and consequently
\begin{equation*}
\left\| W_p^p(\mu, L_n) \right\|_r 
\leq C_{\alpha,p,\beta} \Delta^p 
 \sqrt{ \frac rn}  .
\end{equation*}
 
\noindent $\bullet$ Assume that $\alpha \in (p,2p)$. The first integral in~\eqref{eq:intborne} converges, and  thus 
\begin{eqnarray}
\left\| W_p^p(\mu, L_n) \right\|_r 
&\leq& 
C_{\alpha, p, \beta} \Delta^p 
\left(4^{- k^*p} + \sqrt{\frac rn} + \frac{r}{n} 4^{  k^*(\alpha -p)} \right). \label{eq:p-2p}
\end{eqnarray}
If $n \leq r$, the inequality of Proposition \ref{prop:bound-r-bounded} is obviously true (with $C_{\alpha,p,\beta}=1$). If $n > r$, let
 $x^* =   \frac{\log \frac nr}{ \alpha \log 4}$ such that   $4^{- x^* p} = \frac rn 4^{ x^* (\alpha- p)}$ and let $k^* = \lceil x^* \rceil  $. Then,
\begin{eqnarray*}
\left\| W_p^p(\mu, L_n) \right\|_r 
&\leq& 
C_{\alpha, p, \beta} \Delta^p 
\left( 4^{- x^*p} + \sqrt{\frac rn} + \frac{r}{n} 4^{  (x^*+1)(\alpha -p)} \right) \\
 &\leq& C_{\alpha,p,\beta} \Delta^p   \left( \sqrt{ \frac rn} +  \left( \frac rn\right) ^{p/\alpha} \right)
 \\
&\leq& C_{\alpha,p,\beta} \Delta^p 
 \sqrt{ \frac rn}  \, .
\end{eqnarray*}

\smallskip

\noindent $\bullet$  Assume that $\alpha > 2 p$. The two integrals in \eqref{eq:intborne} now diverge and 
\begin{eqnarray*}
\left\| W_p^p(\mu, L_n) \right\|_r 
&\leq& 
C_{\alpha,p,\beta} \Delta^p
\left(4^{- k^*p} + \sqrt{\frac rn} 4^{  k^*(\alpha/2 -p)} + \frac{r}{n} 4^{  k^*(\alpha -p)} \right).
\end{eqnarray*}
If $n \leq r$, the inequality of Proposition \ref{prop:bound-r-bounded} is obviously true (with $C_{\alpha,p,\beta}=1$). If $n>r$,
let $x^* =   \frac{\log \frac nr}{ \alpha \log 4}$ such that   $4^{- x^* p} = \frac rn 4^{ x^* (\alpha- p)} = \sqrt{\frac rn} 4^{  x^*(\alpha/2 -p)} $ and let $k^* = \lceil x^\star \rceil  $. It follows that 
\begin{equation*}
\left\| W_p^p(\mu, L_n )  \right\|_r 
\leq    C_{\alpha,p,\beta}  \Delta^p  \left( \frac rn\right) ^{p/\alpha}.
\end{equation*}

\noindent $\bullet$  Assume  $\alpha=p$. Then we have  
\begin{eqnarray*}
\left\| W_p^p(\mu, L_n) \right\|_r 
&\leq& 
C_{\alpha,\beta}  \Delta^p
\left(4^{- k^*p} + \sqrt{\frac rn}   + \frac{r}{n}  k^* \log 4 \right) \\
&\leq&  C_{\alpha,\beta}  \Delta^p
\left(4^{- k^*p} + \sqrt{\frac rn}   + \frac{r}{n}  4^{  k^*p/2} \right) \, .
\end{eqnarray*}
This last bound is a particular case of  Inequality~\eqref{eq:p-2p}, leading to a final bound
$C_{\alpha,\beta} \Delta^p 
 \sqrt{ \frac rn}$. 
 
 \smallskip

\noindent $\bullet$  Assume  $  \alpha = 2p $. Then we have  
\begin{eqnarray*}
\left\| W_p^p(\mu, L_n) \right\|_r 
&\leq& 
C_{\alpha,\beta} \Delta^p
\left(4^{- k^*p} + \sqrt{\frac rn}  k^* \log 4  + \frac{r}{n}  \right).
\end{eqnarray*}
If $n \leq r$, the inequality of Proposition \ref{prop:bound-r-bounded} is obviously true (with $C_{\alpha,p,\beta}=1$). If $n>r$, let
$x^* =   \frac{\log \frac nr}{ \alpha \log 4}$ such that   $4^{- x^* p} = \sqrt{\frac rn }$ and let $k^* = \lceil x^* \rceil  $. Then
\begin{eqnarray*}
\left\| W_p^p(\mu, L_n) \right\|_r 
&\leq& 
 C_{\alpha,\beta} \Delta^p
\left(4^{- x^*p} + \sqrt{\frac rn}(x^*+1) \log 4   + \frac rn   \right) \\
&\leq& 
C_{\alpha,\beta} \Delta^p \sqrt{\frac rn} \log  \left (e +\frac nr \right ) \, . 
\end{eqnarray*}
\end{proof}

\subsection{Hoeffding-type inequalities}

Starting from Proposition \ref{prop:bound-r-bounded} and optimizing in $r\geq 2$, we obtain the following concentration inequality under Assumption~\eqref{eq:dimmetric}.
\begin{proposition} \label{PropHoeff}
Under Assumption~\eqref{eq:dimmetric}, 
\[
{\mathbb P} \left(  W_p^p(\mu, L_n)   > x \right)  \leq e^2 \ind_{x \leq \Delta^p}   
\begin{cases}
			\exp \left( -n \left[ \frac{ x}{C_{p,\alpha,\beta}\Delta^p  } \right]^{2} \right)  &    \text{if } \alpha < 2p ; \\
            \exp\left( - n \left[ \frac{ x}{  C_{\alpha,\beta}\Delta^p  \log(e+ \frac{C_{\alpha, \beta}\Delta^p   }{x}  )} \right]^{2} \right)    &   \text{if } \alpha =  2p ; \\            
            \exp \left( - n  \left[ \frac
  {x} {C_{p,\alpha,\beta}\Delta^p   } \right]^{\frac\alpha p} \right)    &   \text{if } \alpha >   2p.
		 \end{cases} 
\]
\end{proposition}

\begin{remark}
In the case where $E={\mathbb R}^d$ and the support $S$ is  bounded, one can always choose  $\alpha=d$ (even if it is not necessarily the optimal choice). For this choice $\alpha=d$, Proposition \ref{PropHoeff} is exactly  Proposition~10 of
\cite{fournier2015rate}. 
\end{remark}

\begin{proof}
$\bullet$ For $\alpha >2p$, we have $\left\| W_p^p(\mu, L_n) \right\|_r\leq C_{\alpha,p,\beta} \Delta^p    \left( \frac rn\right)^{\frac p\alpha}$. Applying Markov's inequality, we get $${\mathbb P}\left(  W_p^p(\mu, L_n)   > x \right) \leq \left(\frac{C_{\alpha,p,\beta} \Delta^p    \left( \frac rn\right)^{\frac p\alpha}}{x}\right)^r.$$
We now choose  $r\geq 2$, such that
$$\Delta^p C_{p,\alpha,\beta}\left(\frac { r} n \right)^{p/\alpha}=\frac x e \, , $$
provided the solution exists in $[2, \infty)$. This leads to the choice 
 $$r=n \left(\frac{x}{C_{p,\alpha,\beta}\Delta^pe}\right)^{\alpha/p} .$$
Consequently, 
\begin{eqnarray*}
{\mathbb P}\left(  W_p^p(\mu, L_n)   > x \right) 
&\leq& 
\exp \left(-
\frac
{n x^{\frac \alpha p}}
{(C_{p,\alpha,\beta}\Delta^p e)^{\frac\alpha p}} \right)
\ind_{\left\{n \left(\frac{x}{C_{p,\alpha,\beta}\Delta^pe}\right)^{\frac \alpha p}\geq 2\right\}} 
+ 
\ind_{\left\{n \left(\frac{x}{C_{p,\alpha,\beta}\Delta^pe}\right)^{\frac \alpha p} < 2\right\}} \\
&\leq& e^2 \exp \left( - n 
\left[ \frac
  {x} 
{C_{p,\alpha,\beta}\Delta^p  } \right]^{\frac\alpha p} \right) .
\end{eqnarray*}

\noindent $\bullet$ For $\alpha <2p$,  we have $\left\| W_p^p(\mu, L_n) \right\|_r \leq C_{\alpha,p,\beta} \Delta^p    \sqrt{ \frac rn}$. Applying Markov's inequality, we get 
$${\mathbb P} \left(  W_p^p(\mu, L_n)   > x \right) \leq \left(\frac{C_{\alpha,p,\beta} \Delta^p    \sqrt{ \frac rn}}{x}\right)^r.$$
We now choose $r \geq 2$ such that
$$\Delta^p C_{p,\alpha,\beta}\sqrt{ \frac rn}  =
\frac x e \, ,$$
provided the solution exists in $[2, \infty)$. This leads to the choice 
$$r=n \left(\frac{x}{C_{p,\alpha,\beta}\Delta^p e}\right)^{2} .$$
Hence, 
\begin{eqnarray*} 
{\mathbb P}\left(  W_p^p(\mu, L_n)   > x \right) 
&\leq& 
\exp \left( -\frac{n x^{2}}{(C_{p,\alpha,\beta}\Delta^pe)^{2}} \right)
\ind_{\left\{n\left(\frac{x}{C_{p,\alpha,\beta}\Delta^pe}\right)^{2}\geq 2\right\}}
+ \ind_{\left\{n\left(\frac{x}{C_{p,\alpha,\beta}\Delta^pe}\right)^{2}< 2\right\}} \\
&\leq& e^2 \exp \left( -n \left[ \frac{ x}{C_{p,\alpha,\beta}\Delta^p  } \right]^{2} \right).
\end{eqnarray*}
\noindent $\bullet$ For $\alpha  = 2p $, we have $\left\| W_p^p(\mu, L_n) \right\|_r \leq C_{\alpha, \beta} \Delta^p    \sqrt{ \frac rn}\log(e +\frac n r)$. Applying Markov's inequality, we get $$ {\mathbb P} \left(  W_p^p(\mu, L_n)   > x \right) \leq \left(\frac{C_{\alpha,\beta} \Delta^p    \sqrt{ \frac rn} \log(e +\frac n r)}{x}\right)^r.$$
Note first that the function $t \rightarrow t \log(e+t^{-1})$ is strictly increasing from $(0, \infty)$ to $(0, \infty)$. We now choose $r \geq 2$ such that
$$ C_{\alpha,\beta} \Delta^p\sqrt{ \frac rn} \log \left (e+  \frac nr \right ) =
\frac x e \, ,$$
provided the solution exists in $[2, \infty)$. The unique solution of this equation satisfies
\begin{equation} \label{solr}
 r = n \left(\frac{x}{C_{\alpha,\beta}\Delta^p e \log(e+  \frac{n}{r}) }\right)^{2} .
 \end{equation}
Hence, for this choice of $r$,
\begin{multline} 
{\mathbb P}\left(  W_p^p(\mu, L_n)   > x \right) 
\leq 
\exp \left( -\frac{n x^{2}}{(C_{\alpha,\beta}\Delta^pe \log(e+  \frac nr) )^{2}} \right)
\ind_{\left\{n\left(\frac{x}{C_{\alpha,\beta}\Delta^pe} \log(e+  \frac nr) \right)^{2}\geq 2\right\}}
+ \ind_{\left\{n\left(\frac{x}{C_{\alpha,\beta}\Delta^pe \log(e+  \frac nr)}\right)^{2}< 2\right\}} \\
\leq e^2 \exp \left( - n \left[ \frac{x}{ C_{\alpha,\beta}\Delta^p e \log(e+  \frac nr)} \right]^{2} \right) \, .  \label{cas2pHoeff1}
\end{multline}
It remains to find a suitable upper bound for the quantity $\log(e+ \frac n r)$, where $r$ is the solution of \eqref{solr}.
If $\frac nr \leq e$, we simply write
\begin{equation} \label{cas2pHoeff2}
\log \left (e+  \frac nr \right ) \leq \log (2 e) \leq 2 \log(e) \leq  2 \log \left( e + \frac{C_{\alpha,\beta}\Delta^p e }{x}  \right) \, .
\end{equation}
Otherwise, we use the fact that $ \log(e + u ) \leq 2 u ^{1/3} $ for $u \geq e$. For $r$ satisfying 
\eqref{solr} and such that $\frac nr \geq e$
\begin{eqnarray*} 
\log \frac nr &=&
2 \log \left(  \frac{C_{\alpha,\beta}\Delta^p e }{x}  \right) +
2 \log \left(  \log\left (e+  \frac nr \right)  \right) \nonumber \\
 &\leq &
2 \log \left( e + \frac{C_{\alpha,\beta}\Delta^p e }{x}  \right) +
2 \log \left( 2 \left(\frac nr \right) ^{1/3}   \right) \nonumber \\
 &\leq & 2 \log \left( e + \frac{C_{\alpha,\beta}\Delta^p e }{x}  \right) +
2 \log  2+ \frac 23 \log  \frac nr    \, ,
\end{eqnarray*}
and  consequently
\begin{equation} \label{cas2pHoeff3}
\log \left(e +\frac nr \right) \leq 2 \log \frac nr \leq  12 \log \left( e + \frac{C_{\alpha,\beta}\Delta^p e }{x}  \right) + 12 \log  2  \leq 24 \log \left( e + \frac{C_{\alpha,\beta}\Delta^p e }{x}  \right)  \, .
\end{equation}
Finally, combining \eqref{cas2pHoeff1},  \eqref{cas2pHoeff2} and \eqref{cas2pHoeff3}, 
\begin{eqnarray*} 
{\mathbb P}\left(  W_p^p(\mu, L_n)   > x \right) 
&\leq& e^2 \exp\left( - n \left[ \frac{ x}{ C_{\alpha,\beta}\Delta^p \log(e+ \frac{C_{\alpha,\beta}\Delta^p  }{x}  )} \right]^{2} \right).  
\end{eqnarray*}
\end{proof}

\subsection{Maximal Inequalities and almost sure results}

The proof of Proposition \ref{prop:bound-r-bounded} is based on a version of Rosenthal's inequality due to  \cite{pinelis1994optimum}, as recalled in Theorem 
\ref{Lem-Ros} of the appendix. This inequality applies to the maximum of partial sums, which means that the inequalities of Propositions \ref{prop:bound-r-bounded} and \ref{PropHoeff} are true with $n^{-1} \sup_{1 \leq k \leq n} k W_p^p(\mu, L_k)$ 
instead of $W_p^p(\mu, L_n)$. 

From the maximal version of Proposition 3, we can deduce almost sure results for the convergence of $L_n$ to $\mu$ in Wasserstein distance. These almost sure results are given in the next Proposition.

\begin{proposition}\label{ASresults}
Under Assumption~\eqref{eq:dimmetric}, the following result holds:   
\begin{itemize}
    \item  If $ \alpha < 2p$, then almost surely
    $ \displaystyle 
    \limsup_{n \rightarrow \infty} \sqrt{\frac{n}{\log \log n }} \ W_p^p(\mu, L_n) \leq C_{p, \alpha, \beta} \Delta^p   \, .
    $
    \item  If $ \alpha = 2p$, then almost surely
    $ \displaystyle 
    \limsup_{n \rightarrow \infty} \sqrt{\frac{n}{\log n \log \log n }} \ W_p^p(\mu, L_n) \leq  C_{\alpha, \beta} \Delta^p   \, .
    $
    \item  If $ \alpha > 2p$, then almost surely
    $ \displaystyle 
    \limsup_{n \rightarrow \infty}\left (\frac{n}{ \log \log n } \right )^{p/\alpha} \ W_p^p(\mu, L_n) \leq C_{p, \alpha, \beta} \Delta^p   \, .
    $
\end{itemize}
The constants $C_{p, \alpha, \beta}$ are the constants given in Proposition \ref{PropHoeff}.
\end{proposition}

\begin{remark} In the case where $E={\mathbb R}^d$ and 
if $d>2p$,  \cite{BaBo13} (see their Theorem 2) proved that, almost surely, 
\begin{equation}\label{BB} 
\beta'_p(d) \int (f_\mu(x))^{(d-p)/d}  \leq 
 \liminf_{n \rightarrow \infty}  n^{p/d} W_p^p(L_n, \mu)  \leq 
 \limsup_{n \rightarrow \infty}  n^{p/d} W_p^p(L_n, \mu) \leq  \beta_p(d) \int (f_\mu(x))^{(d-p)/d}  \, ,
\end{equation}
where $\beta(d), \beta'(d)$ depend only on $d$, and $f_\mu$ is the density of the absolutely continuous part of $\mu$ (note that this result remains true in case of unbounded support, under some moment conditions on $\mu$). In particular, we see that $n^{p/d} W_p^p(L_n, \mu)$ converges almost surely to $0$ if $\mu$ is singular with respect to the Lebesgue measure
on ${\mathbb R}^d$. Now, if the covering dimension $\alpha$ of $S$ is such that $2p< \alpha < d$ (so that $\mu$ is singular with respect to the Lebesgue measure
on ${\mathbb R}^d$), then we get an almost sure rate of convergence of order $\left ( n/  \log \log n  \right )^{p/\alpha}$, which is more precise than the result of \cite{BaBo13}. However, if $\alpha=d$, we get an extra $\log \log n $ term compared to \eqref{BB}. So one can conjecture that, perhaps with mild additional conditions, the term $\log \log n$ of Proposition \ref{ASresults} could be removed in the case where $\alpha >2p$.
\end{remark}

\begin{remark}
In the case where $E={\mathbb R}$, it is well known that 
$W_1(L_n, \mu)= \int |F_n(t)-F(t)| dt $, where $F_n$ and $F$ are respectively the cumulative distribution functions of $L_n$ and $\mu$. 
We can easily deduce from this representation and the compact law of the iterated logarithm in the cotype 2 Banach space ${\mathbb L}^1(dt)$ (see Chapters 8 and 10 in \cite{LT91}) that the sequence 
\begin{equation*}
    \left \{\frac{\sqrt{n}}{\sqrt{ \log \log n}}W_1(L_n, \mu)\right \}_{n \geq 0}
\end{equation*}
is almost surely relatively compact  (in fact, this result remains true in case of unbounded support, under the condition $\int \sqrt{H(t)} dt < \infty$, where $H$ is the tail distribution defined in \eqref{DefTail}). This gives a particular case where the rate of Proposition \ref{ASresults} (case $\alpha < 2p$) cannot be improved.
\end{remark}

\begin{proof}
Let us do the proof in the case $\alpha < 2p$, the other cases being similar. 
Let $\varepsilon >0$. From the maximal version of the inequality given in Proposition \ref{PropHoeff} (case $\alpha < 2p$), we get 
\begin{equation}\label{BK}
\sum_{n \geq 1} \frac 1 n {\mathbb P} \left ( \frac 1 n  \sup_{1 \leq k \leq n} k W_p^p(\mu, L_k) > (C_{p, \alpha, \beta} \Delta^p   + \varepsilon) \sqrt {\frac{\log \log n}{n}} \right ) < \infty \, .
\end{equation}
Let $a>1$. The sequence $(\sup_{1 \leq k \leq n} k W_p^p(\mu, L_k))_{n \geq 1}$ being non-decreasing, \eqref{BK} is equivalent to 
$$
\sum_{N \geq 1}  {\mathbb P} \left (   \sup_{1 \leq k \leq [a^N]} k W_p^p(\mu, L_k) > (C_{p, \alpha, \beta} \Delta^p   + \varepsilon) a^{N/2}\sqrt{\log N} \right ) < \infty \, .
$$
By Borel-Cantelli, we deduce that 
$$
 \limsup_{N \rightarrow \infty} \frac{a^{-N/2}}{\sqrt{\log N }} \sup_{1 \leq k \leq [a^N]} k W_p^p(\mu, L_k) \leq  (C_{p, \alpha, \beta} \Delta^p   + \varepsilon) \ \ \text{almost surely,}
$$
which implies that
$$
 \limsup_{n \rightarrow \infty} \sqrt{\frac{n}{\log \log n }} W_p^p(\mu, L_n) \leq \sqrt a (C_{p, \alpha, \beta} \Delta^p   + \varepsilon) \ \ \text{almost surely.}
$$
This being true for any $\varepsilon >0$ and $a>1$, the result follows in the case $\alpha < 2p$.
\end{proof}

\section{Unbounded support} \label{SecUnbounded}

Let us set a reference point  $x_0 $ in $E$.
In the case where the probability measure $\mu$ has an unbounded support $S$, as  in~\cite{boissard2014mean} we assume that, for any $r>0$, $S \cap B(x_0, r)$ is of finite dimensional type. More precisely, we assume  the following uniform control in terms of covering numbers: 
there exist $\alpha, \beta >0$ such that for any $r>0$ and any $\delta \in (0,2r] $,
\begin{equation}
\label{eq:dimmetric2}
    N(S\cap B(x_0, r), \delta) \leq  \beta \left( \frac {2r} \delta \right) ^\alpha \, .
\end{equation}

Recall that the tail distribution (with respect to the reference point $x_0$) $H$ of the measure $\mu$ has been defined by \eqref{DefTail}, and that the weak moment $\|X\|_{q, w}^q$ of order $q\geq 1$ of a random variable $X$ distributed according to $\mu$ has been defined by \eqref{DefWeakMoment}.
Let us also introduce the quantities, (if they are finite):
$$ I_{\alpha,p} =  \int_0^{+\infty}  H(t)^{\frac{\alpha-p}{\alpha}} t^{p-1} dt \quad \text{and} \quad
I_{2p,p} =  \int_0^{+\infty}  \sqrt{ H(t)} t^{p-1} dt \, .
$$


\subsection{Fuk-Nagaev type inequalities under weak moment assumptions}

\subsubsection{Statement of the theorems}

\begin{theorem}
\label{theo:dev_unbounded-alphageq2p} Let $p\geq 1$, and
assume that \eqref{eq:dimmetric2} is satisfied for some $\alpha > 2p$.
\begin{itemize}
 \item If $\mu$ has a weak moment of order $rp$ with  $r >2$, then, for any $q > r $ and any $x>0$,
\[
{\mathbb P}\left(  W_p^p(\mu, L_n)   > x \right)  \leq
e^2 \exp \left( - n C_{p,\alpha,\beta}  \left[ \frac {x}{  I_{\alpha,p}   } \right]^{\frac\alpha p} \right){\bf 1}_{x \leq C_p\|d(x_0,X)\|_p^p}
+ 
C_{p,q,r}\left\{     \frac { I_{2p,p}^q}{x^q n^{q/2}}  
     +  \frac {\|X\|^{rp}_{rp, w}}{x^r n^{ r-1}} \right\} .
\]
\item If $I_{2p,p}=\int_0^\infty  \sqrt{H(t)} t^{p-1} dt <\infty$, then, for any $x>0$,
\[
{\mathbb P}\left(  W_p^p(\mu, L_n)   > x \right)  \leq  e^2 \exp \left( - n C_{p,\alpha,\beta} 
\left[ \frac
  {x} 
{  I_{\alpha,p}   } \right]^{\frac\alpha p} \right) {\bf 1}_{x \leq C_p\|d(x_0,X)\|_p^p} +   C_p \frac{I_{2p,p}^2}{x^2n}\, .
\]
    \item  If $\mu$ has a weak moment of order $rp$ with  $r \in (\alpha/(\alpha-p),2)$, then, for any $x>0$,
\[
{\mathbb P}\left(  W_p^p(\mu, L_n)   > x \right)  \leq  
e^2 \exp \left( - n C_{p,\alpha,\beta}  \left[ \frac {x}{  I_{\alpha,p}   } \right]^{\frac\alpha p} \right){\bf 1}_{x \leq C_p\|d(x_0,X)\|_p^p}
+ 
 C_{p,r}  \frac{\|X\|^{rp}_{rp, w}}{x^r n^{r-1}}.
\]
    \item If $I_{\alpha,p}=\int_0^\infty   H(t)^{(\alpha-p)/\alpha }  t^{p-1} dt <\infty$, then, for any $x>0$, 
\[
{\mathbb P}\left(  W_p^p(\mu, L_n)   > x \right)  \leq  C_{p,\alpha,\beta}   \frac{I_{\alpha,p}^{\frac{\alpha}{\alpha-p}} }{x^{\frac{\alpha}{\alpha-p}}n^{\frac{p}{\alpha-p}}}.
\]
\item If  $\mu$ has a weak moment of order $rp$ with  $r \in (1,\alpha/(\alpha-p))$, then, for any $x>0$, 
\[
{\mathbb P}\left(  W_p^p(\mu, L_n)   > x \right)  \leq C_{p,\alpha,\beta,r}   \frac{\|X\|^{rp}_{rp, w}}{x^r n^{r-1}}.
\]
\end{itemize}   
\end{theorem}

\begin{theorem}
\label{theo:dev_unbounded-alphaleq2p} 
Let $p\geq 1$, and assume that \eqref{eq:dimmetric2} is satisfied for some $\alpha < 2p$. 
\begin{itemize}
\item If $\mu$ has a weak moment of order $rp$  with $r>2$, then, for any $q  > r$ and any $x>0$,
\[
{\mathbb P}(W_p^p(L_n,\mu) >  x) \leq
 e^2 \exp \left( - n C_{p,\alpha,\beta} \left[ \frac
  {x} { I_{2p,p}  } \right]^{2} \right){\bf 1}_{x \leq C_p\|d(x_0,X)\|_p^p}
+ C_{p,q,r}\left\{   \frac { I_{2p,p}^q }{x^q n^{q/2}}  
     +  \frac {\|X\|^{rp}_{rp, w} }{x^r n^{ r-1}} \right\} .
\]
    \item If $I_{2p,p}=\int_0^\infty  \sqrt{H(t)} t^{p-1} dt <\infty$, then, for any $x>0$,
\[
{\mathbb P}\left(  W_p^p(\mu, L_n)   > x \right)  \leq C_{p,\alpha,\beta}   \frac{I_{2p,p}^2}{x^2n} .
\]
 \item If $\mu$ has a weak moment of order $rp$ with $r \in (1,2) $, then,  for any $x>0$,
\[
{\mathbb P}\left(  W_p^p(\mu, L_n)   > x \right)  \leq   C_{p,\alpha,\beta,r}   \frac{\|X\|^{rp}_{rp, w}}{x^r n^{r-1}}.
\]
\end{itemize}   
\end{theorem}

\begin{theorem}
\label{theo:dev_unbounded-alphaeq2p} 
Let $p\geq 1$, and assume that \eqref{eq:dimmetric2} is satisfied for $\alpha = 2p$. 
\begin{itemize}

\item If $\mu$ has a weak moment of order $rp$ with $r>2$, then, for any $q  > r$ and any $x>0$,
\begin{multline*}
{\mathbb P}\left(  W_p^p(\mu, L_n)   > x \right)  \leq  e^2 \exp\left( - n C_{p,\alpha,\beta} \left[ \frac{ x}{ I_{2p,p}  \log\left(e+ C_{p,\alpha,\beta} \frac{I_{2p,p} }{x}  \right)} \right]^{2} \right){\bf 1}_{x \leq C_p\|d(x_0,X)\|_p^p} \\ +  C_{p,q,r}\left\{   \frac { I_{2p,p}^q }{x^q n^{q/2}}  
     +  \frac {\|X\|^{rp}_{rp, w} }{x^r n^{ r-1}} \right\}.
\end{multline*}
\item If $I_{2p,p}=\int_0^\infty  \sqrt{H(t)} t^{p-1} dt <\infty$, then, for any $x>0$,
\[
{\mathbb P}\left(  W_p^p(\mu, L_n)   > x \right)  \leq  e^2 \exp\left( - n C_{p,\alpha,\beta} \left[ \frac{ x}{ I_{2p,p}  \log\left(e+ C_{p,\alpha,\beta} \frac{I_{2p,p} }{x}  \right)} \right]^{2} \right){\bf 1}_{x \leq C_p\|d(x_0,X)\|_p^p}  +  C_{p,\alpha,\beta}   \frac{I_{2p,p}^2}{x^2n} .
\]

\item  If $\mu$ has a weak moment of order $rp$ with $r \in (1,2)$, then, for any $x>0$,
\[
{\mathbb P}\left(  W_p^p(\mu, L_n)   > x \right)  \leq  
C_{p,\alpha,\beta,r}   \frac{\|X\|^{rp}_{rp, w}}{x^r n^{r-1}}. \]
\end{itemize}   
\end{theorem}

\begin{remark}
In the case where $\mu$ has a weak moment of order $rp$ with $r>2$ and $E=\mathbb R^d$, the results given in Theorems \ref{theo:dev_unbounded-alphageq2p}, \ref{theo:dev_unbounded-alphaleq2p}, and \ref{theo:dev_unbounded-alphaeq2p}  are more precise than Theorem 2.3 in \cite {BehaviorDedeckerMerl19} for two reasons: first, the results are expressed in terms of the covering dimension  $\alpha$ rather than the dimension $d$; second,   the exponential terms involve integrals $I_{\alpha,p}$ rather than weak moments $\|X\|_{rp, w}$. For instance, we infer from Theorem \ref{theo:dev_unbounded-alphageq2p} (case $\alpha >2p$) that 
$$
\limsup_{n \rightarrow \infty}
{\mathbb P}\left( n^{p/\alpha} W_p^p(\mu, L_n)   > x \right)  \leq
e^2 \exp \left( -  C_{p,\alpha,\beta}  \left[ \frac {x}{  I_{\alpha,p}   } \right]^{\frac\alpha p} \right)\, , 
$$
and from Theorem \ref{theo:dev_unbounded-alphaleq2p} (case $\alpha <2p$) that 
$$
\limsup_{n \rightarrow \infty}{\mathbb P}\left( \sqrt n W_p^p(\mu, L_n)   > x \right)  \leq
e^2 \exp \left( -  C_{p,\alpha,\beta}  \left[ \frac {x}{  I_{2p,p}   } \right]^{2} \right)\, .
$$
Moreover, if we only assume that $I_{2p,p} <\infty$, we also infer from Theorem \ref{theo:dev_unbounded-alphaleq2p} that 
\[
{\mathbb P}\left( \sqrt n W_p^p(\mu, L_n)   > x \right)  \leq C_{p,\alpha,\beta}   \frac{I_{2p,p}^2}{x^2} .
\]
Note that, when $p=1$, $I_{2p,p}= \int_0^\infty \sqrt{ H(t)} dt $. Therefore, in the case $p=d=1$, these two last bounds are  consistent with the central limit theorem of \cite{delBarrio99}, who proved that $\sqrt n W_1(\mu_n, \mu)$ converges in distribution as soon as $\int_0^\infty \sqrt{ H(t)} dt < \infty$. In the same paper the authors also prove that $\sqrt n W_1(\mu_n, \mu)$ is stochastically bounded iff $\int_0^\infty \sqrt{ H(t)} dt < \infty$, which proves the relevance of this integral condition. 
\end{remark}

\begin{remark} 
In the case where $\mu$ has a weak moment of order $rp$ with $r \in (1,2)$ and $E=\mathbb R^d$, the results given in Theorems \ref{theo:dev_unbounded-alphageq2p}, \ref{theo:dev_unbounded-alphaleq2p}, and \ref{theo:dev_unbounded-alphaeq2p}  are more precise than Theorem 2.1 of \cite {BehaviorDedeckerMerl19} for two reasons: first, the results are expressed in terms of the covering dimension  $\alpha$ rather than the dimension $d$; second, in case $\alpha >2p$ and 
$r \in (\alpha/(\alpha-p),2)$, our bounds allow to get the large deviations result \eqref{LDbound}
which cannot be deduced from Theorem 2.1 of \cite {BehaviorDedeckerMerl19}.
\end{remark}
 
\subsubsection{Large and moderate deviations under weak moment assumptions}
As immediate consequences of  Theorems~\ref{theo:dev_unbounded-alphageq2p}, \ref{theo:dev_unbounded-alphaleq2p} and \ref{theo:dev_unbounded-alphaeq2p}, we get the following estimates for large and moderate deviations. 

\begin{corollary}
\label{modev_unbounded-alphageq2p} 
Let $p\geq 1$, and assume that \eqref{eq:dimmetric2} is satisfied for some $\alpha > 2p$.
\begin{itemize}
 \item Assume that $\mu$ has a weak moment of order $rp$ with  $r >2$, and let $\rho \in ((\alpha-p)/\alpha, 1]$. Then, for  any $x>0$,
\[
\limsup_{n \rightarrow \infty} n^{r\rho -1}{\mathbb P}\left(  W_p^p(\mu, L_n)   > 
\frac{x}{n^{1-\rho}} \right)  \leq
C_{p,r}\frac {\|X\|^{rp}_{rp, w}}{x^r} \,   .
\]
\item Assume that $I_{2p,p}=\int_0^\infty  \sqrt{H(t)} t^{p-1} dt <\infty$, and let $\rho \in ((\alpha-p)/\alpha, 1]$. Then, for any $x>0$,
\[
\limsup_{n \rightarrow \infty} n^{2\rho -1}{\mathbb P}\left(  W_p^p(\mu, L_n)   > \frac{x}{n^{1-\rho}} \right)  \leq     C_p \frac{I_{2p,p}^2}{x^2}\, .
\]
    \item  Assume that $\mu$ has a weak moment of order $rp$ with  $r \in (\alpha/(\alpha-p),2)$, and let $\rho \in ((\alpha-p)/\alpha, 1]$. Then, for any $x>0$,
\[
\limsup_{n \rightarrow \infty} n^{r\rho -1}{\mathbb P}\left(  W_p^p(\mu, L_n)   > 
\frac{x}{n^{1-\rho}} \right)   \leq 
 C_{p,r}  \frac{\|X\|^{rp}_{rp, w}}{x^r}.
\]
    \item Assume that $I_{\alpha,p}=\int_0^\infty   H(t)^{(\alpha-p)/\alpha }  t^{p-1} dt <\infty$, and let $\rho \in [(\alpha-p)/\alpha, 1]$. Then, for any $x>0$, 
\[
 n^{\frac{\alpha\rho}{\alpha-p} -1} {\mathbb P}\left(  W_p^p(\mu, L_n)   > \frac{x}{n^{1-\rho}} \right)  \leq  C_{p,\alpha,\beta}   \frac{I_{\alpha,p}^{\frac{\alpha}{\alpha-p}} }{x^{\frac{\alpha}{\alpha-p}}}.
\]
\item Assume that  $\mu$ has a weak moment of order $rp$ with  $r \in (1,\alpha/(\alpha-p))$, and let $\rho \in [1/r, 1]$. Then, for any $x>0$, 
\[
 n^{r\rho -1}{\mathbb P}\left(  W_p^p(\mu, L_n)   > 
\frac{x}{n^{1-\rho}} \right)   \leq C_{p,\alpha,\beta,r}   \frac{\|X\|^{rp}_{rp, w}}{x^r}.
\]
\end{itemize}   
\end{corollary}

\begin{corollary}
\label{modev_unbounded-alphaleq2p} 
Let $p\geq 1$, and assume that \eqref{eq:dimmetric2} is satisfied for some  $\alpha \leq  2p$. 
\begin{itemize}
\item Assume that  $\mu$ has a weak moment of order $rp$  with $r>2$, and let $\rho \in (1/2, 1]$. Then, for any  $x>0$,
\[
\limsup_{n \rightarrow \infty} n^{r\rho -1} {\mathbb P}\left(  W_p^p(\mu, L_n)   > 
\frac{x}{n^{1-\rho}} \right) \leq
  C_{p,r}  \frac {\|X\|^{rp}_{rp, w} }{x^r}  .
\]
    \item Assume that  $I_{2p,p}=\int_0^\infty  \sqrt{H(t)} t^{p-1} dt <\infty$, and let $\rho \in (1/2, 1]$. Then, for any $x>0$,
\[
\limsup_{n \rightarrow \infty} n^{2\rho -1} 
{\mathbb P}\left(  W_p^p(\mu, L_n)   > 
\frac{x}{n^{1-\rho}} \right)  \leq C_{p,\alpha,\beta}   \frac{I_{2p,p}^2}{x^2} .
\]
 \item Assume that $\mu$ has a weak moment of order $rp$ with $r \in (1,2) $ and let $\rho \in [1/r, 1]$. Then,  for any $x>0$,
\[
 n^{r\rho -1} {\mathbb P} \left(  W_p^p(\mu, L_n)   > 
\frac{x}{n^{1-\rho}} \right) \leq   C_{p,\alpha,\beta,r}   \frac{\|X\|^{rp}_{rp, w}}{x^r}.
\]
\end{itemize}   
\end{corollary}

\begin{remark}
Taking $\rho=1$ in Corollaries \ref{modev_unbounded-alphageq2p} and \ref{modev_unbounded-alphaleq2p}, we obtain the large deviations result \eqref{LDbound} described in the introduction.
\end{remark}





 \subsection{Bernstein-type inequalities under sub/super exponential moments}
 
 In this section, we assume that the distribution $\mu$ has a sub/super exponential moment, that is: there exist $\lambda>0, \kappa>0$ such that 
 \begin{equation}\label{FGexp}
   {\mathcal E}_{\kappa, \lambda}(\mu)= \int   \exp ( \lambda d(x_0, x)^\kappa) \mu(dx) < \infty \, .
 \end{equation}
 In this case, we obtain a complete extension of Theorem 1 of \cite{fournier2015rate}, in which the covering dimension $\alpha$ replaces the dimension of the ambient space (the space ${\mathbb R}^d$ in \cite{fournier2015rate}). 
 
 \begin{theorem}\label{theo:dev_unbounded-expo}
Let $p\geq 1$, and assume that \eqref{eq:dimmetric2} is satisfied. 
\begin{enumerate}
\item If \eqref{FGexp} is satisfied for some $\kappa >p $, then, for any $x>0$,
$$
{\mathbb P}\left ( W_p^p(\mu, L_n)>x \right ) \leq A(n,x) 
+ C_1\exp\left (-C_2 n x^{\frac{\kappa}{p}}\right ){\bf 1}_{x > 1} \, , 
$$
where the positive constants $C_1, C_2$ depend on $p, \kappa, \lambda, {\mathcal E}_{\kappa, \lambda}(\mu)$, and 
$$
A(n,x) =  
\begin{cases}
			   & e^2 \exp \left( - n C_{p,\alpha,\beta} \left[ \frac{x} { I_{2p,p}  } \right]^{2} \right){\bf 1}_{x \leq C_p \|d(x_0, X)\|_p^p}  \ \text{ if } \alpha <   2p  \text{ and } I_{2p,p} < \infty; \\
              &  e^2 \exp\left( - n C_{p,\alpha,\beta} \left[ \frac{ x}{ I_{2p,p}  \log\left(e+ C_{p,\alpha,\beta} \frac{I_{2p,p} }{x}  \right)} \right]^{2} \right){\bf 1}_{x \leq C_p\|d(x_0, X)\|_p^p} \  \text{ if } \alpha =  2p \text{ and } I_{2p,p} < \infty ; \\ 
              & e^2 \exp \left( - n C_{p,\alpha,\beta} 
\left[ \frac
  {x} 
{  I_{\alpha,p}   } \right]^{\frac\alpha p} \right){\bf 1}_{x \leq C_p\|d(x_0, X)\|_p^p} \   \text{ if } \alpha >   2p  \text{ and } I_{\alpha,p} < \infty .
		 \end{cases} 
$$
\item If \eqref{FGexp} is satisfied for some $\kappa \in(0,p) $, then, for any $\varepsilon \in (0,\kappa)$ and any $x>0$,
$$
{\mathbb P}\left ( W_p^p(\mu, L_n)>x \right ) \leq A(n,x) 
+ C_1\exp\left (-C_2 (n x)^{\frac{(\kappa-\varepsilon)}{p}}\right ){\bf 1}_{x \leq  1} + C_1\exp\left (-C_2 (n x)^{ \frac \kappa p}\right ){\bf 1}_{x > 1} \, , 
$$
where the positive constants $C_1, C_2$ depend on ${p, \kappa, \lambda, \varepsilon, {\mathcal E}_{\kappa, \lambda}(\mu)}$, and $A(n,x)$ is defined as in Item 1. 

\end{enumerate}

 \end{theorem}

\section{Proofs in the unbounded cases}\label{Sec4}

\subsection{Preliminary results}
\label{subs-prelim}

As a preliminary remark, let us note that we can always assume that $\|d(x_0, X)\|_p=1$. To see this, we first make the change of distance $d'=d/\|d(x_0, X)\|_p$ (so that $\|d'(x_0, X)\|_p=1$). Then, with obvious notations,
$$
{\mathbb P}\left(  W_p^p(\mu, L_n, d)   > x \right)= {\mathbb P}\left(  W_p^p(\mu, L_n, d')   > \frac{x}{\|d(x_0, X)\|_p^p} \right)
\, .
$$
Assume now that Theorems~\ref{theo:dev_unbounded-alphageq2p} to \ref{theo:dev_unbounded-expo} 
are true for $W_p^p(\mu, L_n, d')$. Then the results follow for $W_p^p(\mu, L_n, d)$ by using the elementary homogeneity rules: for any $\delta>0$
$$
\|d(x_0, X)\|_p^p \int_0^\infty t^{p-1} H(t\|d(x_0, X)\|_p)^\delta dt = \int_0^\infty t^{p-1} H(t)^\delta dt \, , 
$$
and for any $r>0$,
$$
\frac{\sup_{t>0} t^{rp} H(t \|d(x_0, X)\|_p)}{\left(\frac{x}{\|d(x_0, X)\|_p^p}\right )^r}= \frac{\|X\|^{rp}_{rp, w}}{x^r}\, .
$$

Hence, in the rest of the proof, we will assume without loss of generality that $\|d(x_0, X)\|_p=1$.

\medskip

The first part of the proofs of 
Theorems~\ref{theo:dev_unbounded-alphageq2p} to \ref{theo:dev_unbounded-expo}
is the same, and it follows  the proof of Corollary 1.3 in~\cite{boissard2014mean}. We adopt their notation.

Let $r_i = 2^i$, $i \geq 1$. We introduce the rings 
$$K_1 = B(x_0, r_1), \hskip 1cm   K_i = B(x_0, r_i) \setminus B(x_0, r_{i_{1-1}})$$
and define on these the conditional measures $\mu^{K_i} = \frac{\ind_{K_i} \mu}{\mu(K_i)} $ and  $L_n^i = \frac{\ind_{K_i}  L_n}{L_n(K_i)} $.
Let also
$$ \lambda = \sum_{i \geq 1}   \left(L_n(K_i) - \mu(K_i) \right)\vee 0 =  \sum_{i \geq 1} \left(\mu(K_i) - L_n(K_i)  \right)\vee 0 ,$$
 $$\lambda_i = L_n(K_i) \wedge \mu(K_i), $$
 and then note that
$$ \lambda = \sum_{i \geq 1}   \left(L_n(K_i) - \mu(K_i) \right)\vee 0 =  \sum_{i \geq 1} \left(\mu(K_i) - L_n(K_i)  \right)\vee 0 .$$ 

We introduce the two measures   
 $$\gm =\frac 1 \lambda \sum_{i \geq 1}   \left[\left(L_n(K_i) - \mu(K_i) \right)\vee 0 \right] L_n^i$$ 
 $$ \gn = \frac 1 \lambda \sum_{i \geq 1}   \left[\left(\mu(K_i) - L_n(K_i)  \right)\vee 0 \right] \mu^{K_i}, $$
 so that 
 $$L_n = \sum_{i \geq 1} \lambda_i L_n^i + \lambda  \gm$$
 and
 $$ \mu = \sum_{i \geq 1} \lambda_i \mu^{K_i} + \lambda  \gn.$$
 
 Here, we need the following lemma,  adapted from \cite{boissard2014mean}. The original version (Lemma 2.4 in \cite{boissard2014mean}) is stated for $W_p$, but the proof for the control of $W_p^p$ is identical.
\begin{lemma}
\label{lem:2dot4BL}
 Let $1 \leq  p < \infty$, $\mu, \, \nu \in  \mathcal{P}_p(E)$ and for $i \geq 1$, let $\mu_i, \, \nu_i \in \mathcal{P}_p(E)$, $\lambda_i \geq 0$ be such that
$\sum_{i \geq 1}  \lambda_i = 1$ and
$\sum_{i \geq 1} \lambda_i \mu_i = \mu$, $\sum_{i \geq 1} \lambda_i \nu_i = \nu$. 
Then the following bound holds:
$$ W_p^p (\mu, \nu) \leq \sum_{i \geq 1} \lambda_i W_p^p (\mu_i , \nu_i ) .$$
\end{lemma} 
 
 According to Lemma~\ref{lem:2dot4BL}, we have 
 \begin{eqnarray*}
 W_p^p(L_n,\mu) &\leq &  \sum_{i \geq 1} \lambda_i W_p^p (L_n^i  , \mu^{K_i}  ) + \lambda W_p^p (\gm  , \gn  )\\
 &\leq & \sum_{i \geq 1} \lambda_i W_p^p (L_n^i  , \mu^{K_i}  ) + 2^{p-1} \lambda W_p^p (\gm  , \delta_{x_0}  ) + 2^{p-1} \lambda W_p^p (\gn  , \delta_{x_0}  ).
 \end{eqnarray*}
Thus for $x>0$ we get
\begin{equation}
\label{borne3termes-bis}
{\mathbb P}(W_p^p(L_n,\mu) > 3 x) \leq P \left(\sum_{i \geq 1} \lambda_i W_p^p (L_n^i  , \mu^{K_i}  ) \geq x \right)    +    P \left( 2^{p-1} \lambda W_p^p (\gm  , \delta_{x_0} )  \geq x \right) +    P \left( 2^{p-1} \lambda W_p^p (\gn  , \delta_{x_0}  )  \geq x \right).
\end{equation}
In the rest of the proof, we will call the first term of upper bound ~\eqref{borne3termes-bis} the {\it main term} and the other two terms the {\it remainder terms}. All the different cases will start from Inequality~\eqref{borne3termes-bis}.

We now give a key lemma to provide exponential deviations for the main term. 
\begin{lemma}
\label{expobound-mainterm}
Using the notation introduced above,

$$
{\mathbb P} \left(\sum_{i \geq 1} \lambda_i W_p^p (L_n^i  , \mu^{K_i}  ) \geq x \right)   \leq 
\begin{cases}
			   & e^2 \exp \left( - n C_{p,\alpha,\beta} \left[ \frac
  {x} 
{ I_{2p,p}  } \right]^{2} \right){\bf 1}_{x \leq C_p}  \ \text{ if } \alpha <   2p  \text{ and } I_{2p,p} < \infty; \\
              &  e^2 \exp\left( - n C_{p,\alpha,\beta} \left[ \frac{ x}{ I_{2p,p}  \log\left(e+ C_{p,\alpha,\beta} \frac{I_{2p,p} }{x}  \right)} \right]^{2} \right){\bf 1}_{x \leq C_p} \  \text{if } \alpha =  2p \text{ and } I_{2p,p} < \infty ; \\ 
              & e^2 \exp \left( - n C_{p,\alpha,\beta} 
\left[ \frac
  {x} 
{  I_{\alpha,p}   } \right]^{\frac\alpha p} \right){\bf 1}_{x \leq C_p} \   \text{ if } \alpha >   2p  \text{ and } I_{\alpha,p} < \infty .
		 \end{cases} 
$$		 
\end{lemma}

\begin{proof}[Proof of Lemma~\ref{expobound-mainterm}]
\noindent $\bullet$ Assume that $\alpha >   2p$  and  $I_{\alpha,p} < \infty$. We apply to each term $\lambda_i W_p^p (L_n^i  , \mu^{K_i}  )$ our results on the bounded case (see Proposition~\ref{prop:bound-r-bounded}), that is on $K_i$.  For a moment of order $q \geq 2$ and since $\alpha > 2p$, applying Proposition~\ref{prop:bound-r-bounded} conditionally to the sample size $n L_n(K_i)$, we get that 
 \begin{eqnarray}
\| \lambda_i W_p^p (L_n^i  , \mu^{K_i}  ) \|_q   &\leq&  \left \|\lambda_i  r_i^p   C_{\alpha,p,\beta}  \left( \frac q {n L_n(K_i)} \right )^{\frac{p}{\alpha}} \right \|_q \notag \\
&\leq&   \left \|  r_i^p   C_{\alpha,p,\beta}  \left( \frac q {n L_n(K_i)} \right )^{\frac{p}{\alpha}}  L_n(K_i)^{\frac{p}{\alpha}} \mu(K_i)   ^{\frac{\alpha-p}{\alpha}} \right \|_q \notag \\
&\leq&     r_i^p   C_{\alpha,p,\beta}  \left( \frac q n \right )^{\frac{p}{\alpha}}  \mu(K_i)   ^{\frac{\alpha-p}{\alpha}} \label{ineq:conditnLKi}
\end{eqnarray}
where we have used the definition of  $\lambda_i$ for the second inequality. Next, we apply Markov's inequality at order $q$: 
\begin{eqnarray}
{\mathbb P} \left(\sum_{i \geq 1} \lambda_i W_p^p (L_n^i  , \mu^{K_i}  )  \geq 2x \right) 
& \leq & \frac{1}{x^q}\left ( \sum_{i \geq 1} \| \lambda_i W_p^p (L_n^i  , \mu^{K_i}  ) \|_q  \right )^q \nonumber \\ 
& \leq & \frac{2^{q-1}}{x^q} \| \lambda_1 W_p^p (L_n^1  , \mu^{K_1}  ) \|_q^q + \frac{2^{q-1}}{x^q}  \left( \sum_{i \geq 2} \| \lambda_i W_p^p (L_n^i  , \mu^{K_i}  ) \|_q  \right )^q \, . \label{tricky}
\end{eqnarray}
To deal with the first term on the right-hand side in \eqref{tricky}, we write that 
$$
\frac{2^{q-1}}{x^q} \| \lambda_1 W_p^p (L_n^1  , \mu^{K_1}  ) \|_q^q \leq \frac {C_{\alpha,p,\beta}^q}{x^q} \left( \frac q n \right )^{\frac{pq}{\alpha}}\leq \frac {C_{\alpha,p,\beta}^q}{x^q} \left( \frac q n \right )^{\frac{pq}{\alpha}} \left(\int_0^{+\infty}  H(t)^{\frac{\alpha-p}{\alpha}} t^{p-1} dt    \right)^q \, ,
$$
the second inequality being true because the integral $I_{\alpha, p}$ is greater than $\|d(x_0, X)\|_p^{p}=1$.
To deal with the second term on the right-hand side in \eqref{tricky}, we first note that, for any $i\geq 2$, we have $   \mu(K_i)  \leq   H(r_{i-1}) \leq H(t) $ for any $t \in [r_{i-2},r_{i-1}]$. Moreover $r_i^p = p  \int_{r_{i-1}}^{r_i} t^{p-1} dt +  r_{i-1}^p$ so $ r_i^p = p (1-2^{-p})^{-1} \int_{r_{i-1}}^{r_i} t^{p-1} dt$. It gives that
$$
 \frac{2^{q-1}}{x^q}  \left( \sum_{i \geq 2} \| \lambda_i W_p^p (L_n^i  , \mu^{K_i}  ) \|_q  \right )^q \leq 
\frac {C_{\alpha,p,\beta}^q}{x^q} \left( \frac q n \right )^{\frac{pq}{\alpha}} \left(\int_0^{+\infty}  H(t)^{\frac{\alpha-p}{\alpha}} t^{p-1} dt    \right)^q \, .
$$
It follows from \eqref{tricky} and the preceding computations  that
\begin{equation}
\label{eq:proba-optim-q}
{\mathbb P} \left(\sum_{i \geq 1} \lambda_i W_p^p (L_n^i  , \mu^{K_i}  ) \geq x \right)   \leq \frac {C_{\alpha,p,\beta}^q}{x^q} \left( \frac q n \right )^{\frac{pq}{\alpha}} \left(\int_0^{+\infty}  H(t)^{\frac{\alpha-p}{\alpha}} t^{p-1} dt    \right)^q  ,
\end{equation}
and by optimizing  in $q \geq 2$ as in the proof of Proposition \ref{PropHoeff} (case $\alpha > 2p$), 
we find that  
\begin{eqnarray}\label{OUF!!}
{\mathbb P} \left(\sum_{i \geq 1} \lambda_i W_p^p (L_n^i  , \mu^{K_i}  ) \geq x \right)   
&\leq&  e^2 \exp \left( - n 
\left[ \frac
  {x} 
{C_{p,\alpha,\beta} I_{\alpha,p} e } \right]^{\frac\alpha p} \right) .
\end{eqnarray}
Moreover, we have the obvious upper bound 
\begin{equation}\label{trivial}
\left \|\sum_{i \geq 1} \lambda_i W_p^p (L_n^i  , \mu^{K_i}  ) \right \|_\infty \leq 2^p\sum_{i \geq 1} 2^{ip}\mu(K_i) = C_p < \infty \, ,
\end{equation}
the last inequality being true because we have assumed that $\|d(x_0, X)\|_p=1$. Hence the indicator ${\bf 1}_{x \leq C_p}$ can be added in the upper bound \eqref{OUF!!}.


\medskip

\noindent $\bullet$  Assume that $\alpha <   2p$  and  $I_{2p,p} < \infty$.  As in \eqref{ineq:conditnLKi}, we first apply Proposition~\ref{prop:bound-r-bounded}   to each term of the sum conditionally to the sample size $n L_n(K_i)$. For  $q \geq 2$ and $\alpha < 2p$, it gives that
\begin{equation} \label{ineq:conditnLKi2}
\| \lambda_i W_p^p (L_n^i  , \mu^{K_i}  ) \|_q   \leq     r_i^p   C_{\alpha,p,\beta}  \sqrt{ \frac q n }  \sqrt{\mu(K_i)}  .
\end{equation}
Using \eqref{tricky} and proceeding as for the case $\alpha >2p$, we obtain that 
\begin{equation}
\label{optq-CaseB}
{\mathbb P} \left(\sum_{i \geq 1} \lambda_i W_p^p (L_n^i  , \mu^{K_i}  ) \geq x \right)   \leq \frac {C_{\alpha,p,\beta}^q}{x^q} \left( \frac q n \right )^{\frac{q}{2}} \left(\int_0^{+\infty}  \sqrt{  H(t)} t^{p-1} dt    \right)^q  ,
\end{equation}
and by optimizing  in $q \geq 2$ as in the proof of Proposition \ref{PropHoeff} (case $\alpha < 2p$), 
we find that  
\begin{eqnarray}\label{OUF2!!}
{\mathbb P} \left(\sum_{i \geq 1} \lambda_i W_p^p (L_n^i  , \mu^{K_i}  ) \geq x \right)   
&\leq&  e^2 \exp \left( - n 
\left[ \frac
  {x} 
{C_{p,\alpha,\beta} I_{2p,p} e } \right]^{2} \right) .
\end{eqnarray}
Using \eqref{trivial}, we see that the indicator ${\bf 1}_{x \leq C_p}$ can be added in the upper bound \eqref{OUF2!!}.

\medskip

\noindent $\bullet$ Assume that   $\alpha = 2p$ and  $I_{2p,p} < \infty$. As in \eqref{ineq:conditnLKi}, we first apply Proposition~\ref{prop:bound-r-bounded} to each term of the sum conditionally to the sample size $n L_n(K_i)$. For  $q \geq 2$ and $\alpha = 2p$, it gives that 
$$
\| \lambda_i W_p^p (L_n^i  , \mu^{K_i}  ) \|_q  \leq 
\left \| \lambda_i r_i^p C_{\alpha,p,\beta} \sqrt{ \frac q {nL_n(K_i)} } \log \left ( e + \frac{nL_n(K_i)}{q} \right ) \right \|_q
\leq 
\left \|\lambda_i r_i^p C_{\alpha,p,\beta} \sqrt{ \frac q {nL_n(K_i)} } \log \left ( e + \frac{n}{q} \right )  \right \|_q\, .
$$
Since $\lambda_i \leq \sqrt{L_n(K_i)}\sqrt{\mu(K_i)}$, we get
$$
\| \lambda_i W_p^p (L_n^i  , \mu^{K_i}  ) \|_q  \leq 
r_i^p   C_{\alpha,p,\beta}  \sqrt{ \frac q n } \log \left ( e + \frac{n}{q} \right )  \sqrt{\mu(K_i)} \, .
$$
Using \eqref{tricky} and proceeding as for the case $\alpha >2p$, we obtain that 
\begin{equation}
\label{optq-CaseC}
{\mathbb P} \left(\sum_{i \geq 1} \lambda_i W_p^p (L_n^i  , \mu^{K_i}  ) \geq x \right)   \leq \frac {C_{\alpha,p,\beta}^q}{x^q} \left (\sqrt{ \frac q n } \log \left ( e + \frac{n}{q} \right )\right )^q \left(\int_0^{+\infty}  \sqrt{  H(t)} t^{p-1} dt    \right)^q  ,
\end{equation}
and by optimizing  in $q \geq 2$ as in the proof of Proposition \ref{PropHoeff} (case $\alpha=2p$), we find that  
\begin{eqnarray}\label{OUF3!!}
{\mathbb P} \left(\sum_{i \geq 1} \lambda_i W_p^p (L_n^i  , \mu^{K_i}  ) \geq x \right)    
  &\leq& e^2 \exp\left( - n \left[ \frac{ x}{24 C_{p,\alpha,\beta}I_{2p,p} e \log(e+ \frac{C_{p,\alpha,\beta}I_{2p,p} e }{x}  )} \right]^{2} \right).  
\end{eqnarray}
Using \eqref{trivial}, we see that the indicator ${\bf 1}_{x \leq C_p}$ can be added in the upper bound \eqref{OUF3!!}.
\end{proof} 
 
We are now in a position to prove each of the cases covered by the four theorems.

 \medskip 
 

\subsection{Proof of Theorem~\ref{theo:dev_unbounded-alphageq2p}}

Recall that  $\alpha > 2p$, and that we assume that $\|d(x_0, X)\|_p=1$.

\smallskip

\noindent $\bullet$ Let's start with the case where $\mu$ has weak moments of order $pr$ with $r >  2$.  
We first control the main term of the upper bound \eqref{borne3termes-bis}. We have that $\sup_{t \geq 0} t^p H(t)^{1/r} < \infty$. Since $1/r < 1/2 <     (\alpha-p)  / \alpha $, then $ I_{\alpha,p} = \int t^{p-1} H(t)^{(\alpha-p)/\alpha} dt < \infty$. We can then directly apply Lemma~\ref{expobound-mainterm}:
\begin{equation}
    \label{eq:MainTermcaseA-bis}
{\mathbb P} \left(\sum_{i \geq 1} \lambda_i W_p^p (L_n^i  , \mu^{K_i}  ) \geq x \right)  
\leq e^2 \exp \left( - n C_{p,\alpha,\beta} 
\left[ \frac
  {x} 
{  I_{\alpha,p}   } \right]^{\frac\alpha p} \right) {\bf 1}_{x \leq C_p}.
\end{equation}
We now consider the remainder terms in \eqref{borne3termes-bis}. Note that, by Lemma \ref{lem:2dot4BL}, 
\begin{equation}\label{Lemma1again}
\lambda W_p^p (\gm  , \delta_{x_0} )   \leq \sum_{i \geq 1}  [(L_n(K_i) - \mu(K_i)) \vee 0  ] r_i^p \quad  \text{and} \quad
\lambda W_p^p (\gn  , \delta_{x_0} )   \leq \sum_{i \geq 1}  [(\mu(K_i) - L_n(K_i)) \vee 0  ] r_i^p \, .
\end{equation} 
Starting from \eqref{Lemma1again} and applying Markov's inequality, we get that: for any integer $K\geq 2$ and any $q > r$, 
 \begin{eqnarray} 
 {\mathbb P} \left(  \lambda W_p^p (\gm  ,\delta_{x_0} )  \geq 2 x \right)  &\leq & 
  {\mathbb P} \left( \sum_{i = 1}^K  [(L_n(K_i) - \mu(K_i)) \vee 0  ] r_i^p    \geq   x \right) +  {\mathbb P} \left(  \sum_{i =K+ 1}^\infty  [(L_n(K_i) - \mu(K_i)) \vee 0  ] r_i^p     \geq   x \right) \nonumber \\ \label{Mq}
  &\leq &  \frac 1{x^q} \left\| \sum_{i = 1}^K  [(L_n(K_i) - \mu(K_i)) \vee 0  ] r_i^p     \right\|_q^q +  \frac 1 x \left\|  \sum_{i =K+ 1}^\infty  [(L_n(K_i) - \mu(K_i)) \vee 0  ] r_i^p    \right\|_1 \nonumber \\
    &\leq &  \frac 1{x^q} \left(  \sum_{i = 1}^K \left\| L_n(K_i) - \mu(K_i) \right\|_q r_i^p   \right)^q +  \frac 1 x  \sum_{i =K+ 1}^\infty  \left\|   L_n(K_i) - \mu(K_i) \right\|_1 r_i^p \, . 
\end{eqnarray}
Let us deal with the first term on the right-hand side in \eqref{Mq}. We have 
\begin{equation}\label{premdec}
\frac 1{x^q} \left(  \sum_{i = 1}^K \left\| L_n(K_i) - \mu(K_i) \right\|_q r_i^p   \right)^q   \leq \frac {2^{q-1}}{x^q} \left(   \left\| L_n(K_1) - \mu(K_1) \right\|_q 2^p   \right)^q +   \frac {2^{q-1}}{x^q} \left(  \sum_{i = 2}^K \left\| L_n(K_i) - \mu(K_i) \right\|_q r_i^p   \right)^q  \, .
\end{equation}
For the first term on the right-hand side in \eqref{premdec} we use Burkholder inequality, as recalled in Theorem \ref{Burkholder} of the appendix. The variables ${\bf 1}_{X_i \in K_1}- \mu(K_1) $ being bounded by 1, this gives
\begin{equation}\label{normq1}
\frac {2^{q-1}}{x^q} \left(   \left\| L_n(K_1) - \mu(K_1) \right\|_q 2^p \right)^q\leq \frac {C_{p,q}}{x^q n^{q/2}} \leq \frac {C_{p,q}}{x^q n^{q/2}} \left(\int_0^\infty \sqrt{H(t)} t^{p-1} dt    \right)^q \, ,
\end{equation}
the last inequality being true because we have assumed that $\|d(x_0,X)\|_p=1$. 
For the second term on the right-hand side in \eqref{premdec}, we apply Rosenthal's inequality, as recalled in Theorem \ref{Lem-Ros} of the appendix.  This gives
\begin{equation}\label{normq2}
\frac {2^{q-1}}{x^q} \left(  \sum_{i = 2}^K \left\| L_n(K_i) - \mu(K_i) \right\|_q r_i^p   \right)^q  \leq 
\frac {C_{p,q}}{x^q n^{q/2}} \left(  \sum_{i = 2}^K   \sqrt{ \mu(K_i)} r_i^p \right)^q  
     + \frac {C_{p,q}}{x^q n^{ q-1}} \left(  \sum_{i = 2}^K   \mu(K_i)^{\frac 1q}    r_i^p  \right)^q \, .
\end{equation}
For the second term in \eqref{Mq}, we simply note that 
\begin{equation}\label{norm1}
\frac 1 x  \sum_{i =K+ 1}^\infty  \left\|   L_n(K_i) - \mu(K_i) \right\|_1 r_i^p \leq \frac 2 x  \sum_{i =K+ 1}^\infty  \mu(K_i)  r_i^p \, .
\end{equation}
Combining the inequalities \eqref{Mq}, \eqref{premdec}, \eqref{normq1}, \eqref{normq2},  \eqref{norm1}, and comparing the sums with integrals (as done in the proof of Lemma \ref{expobound-mainterm}), we get that, for any $M>0$, 
 \begin{equation} \label{caseATermeRestBound-bis}
 {\mathbb P} \left(  \lambda W_p^p (\gm  , \delta_{x_0} )  \geq 2 x \right) 
      \leq   \frac {C_{p,q}}{x^q n^{q/2}} \left(\int_0^\infty \sqrt{H(t)} t^{p-1} dt    \right)^q  
     + \frac { C_{p,q} }{x^q n^{ q-1}} \left(\int_0^M  H(t)^{\frac1q} t^{p-1} dt    \right)^q     +  \frac {   C_p }{x  }  \int_ M^\infty   H(t)  t^{p-1} dt . 
\end{equation}
Since  $\mu$ satisfies a weak moment of order $rp$, then $H(t)  \leq 1 \wedge t^{-rp} \|X\|^{rp}_{rp, w}$, which gives  $I_{2p,p} = \int_0^\infty  \sqrt{H(t)} t^{p-1} dt <\infty$. We also use this bound on $H$ to upper bound and balance the two last terms in the right hand of \eqref{caseATermeRestBound-bis}. We take   $M = (xn)^{1/p}$ and then
  \begin{equation} 
  \label{caseA:rest-bis}
 {\mathbb P} \left(  \lambda W_p^p (\gm  , \mu )  \geq 2 x \right) 
      \leq    C_{p,q,r}\left\{   \frac { I_{2p,p}^q }{x^q n^{q/2}}  
     + \frac {  \|X\|^{rp}_{rp, w}}{x^r n^{ r-1}}   \right\},
   \end{equation}
where  $C_{p,q,r}$ only depends $p$, $q$ and $r$.

Starting from \eqref{Lemma1again}, it can be checked that the same inequality is valid for ${\mathbb P} \left(  \lambda W_p^p (\gn  , \mu )  \geq 2 x \right)$. We conclude by combining Equations~\eqref{eq:MainTermcaseA-bis} and~\eqref{caseA:rest-bis}.

\smallskip

\noindent $\bullet$ Let's now assume that $\mu$ has a weak moment of order $rp$ with  $r \in (\alpha/(\alpha-p),2)$. As in the previous case, we have that $ \sup_{t \geq 0} t^p H(t)^{1/r} < \infty$, with  $1/r <     (\alpha-p)  / \alpha $ and thus the upper bound~\eqref{eq:MainTermcaseA-bis} is still satisfied, which gives the control of the main term of \eqref{borne3termes-bis}.
 
We now consider the remainder terms of \eqref{borne3termes-bis}. Starting from \eqref{Lemma1again}, we have
\begin{multline}\label{pain}
{\mathbb P} \left(  \lambda W_p^p (\gm  ,\delta_{x_0} )  \geq 3 x \right)  \leq  {\mathbb P} \left(   [(L_n(K_1) - \mu(K_1)) \vee 0  ] r_i^p    \geq   x \right) +
  {\mathbb P} \left( \sum_{i = 2}^K  [(L_n(K_i) - \mu(K_i)) \vee 0  ] r_i^p    \geq   x \right) \\ +  {\mathbb P} \left(  \sum_{i =K+ 1}^\infty  [(L_n(K_i) - \mu(K_i)) \vee 0  ] r_i^p     \geq   x \right) \, .
\end{multline}
For the first term on the right-hand side in \eqref{pain}, we apply Markov's inequality at order $r$. This gives 
\begin{equation}\label{term1}
{\mathbb P} \left(   [(L_n(K_1) - \mu(K_1)) \vee 0  ] r_i^p    \geq   x \right) \leq \frac{C_{p,r}}{x^r} \|L_n(K_1)-\mu(K_1)\|_2^r \leq \frac{C_{p,r}}{x^r n^{r/2}} \leq C_{p,r} \frac{\|X\|^{rp}_{rp, w}}{x^r n^{r-1}} \, ,
\end{equation}
the last inequality being true because we have assumed that $\|d(x_0,X)\|_p=1$ (implying that $\|X\|^{rp}_{rp, w} \geq C_{p,r}>0$), and because $r/2>r-1$. For the second term on the right-hand side in \eqref{pain}, we apply Markov's inequality at order 2. This gives 
\begin{equation}\label{term2}
{\mathbb P} \left( \sum_{i = 2}^K  [(L_n(K_i) - \mu(K_i)) \vee 0  ] r_i^p    \geq   x \right) \leq \frac {2}{x^2} \left(  \sum_{i = 2}^K \left\| L_n(K_i) - \mu(K_i) \right\|_2 r_i^p   \right)^2  \leq 
\frac {2}{x^2  n} \left(  \sum_{i = 2}^K   \sqrt{ \mu(K_i)} r_i^p \right)^2  
      \, ,
\end{equation}
For the third term on the right-hand side in \eqref{pain}, we use the upper bound given in \eqref{norm1}. 
Combining the inequalities \eqref{pain}, \eqref{term1}, \eqref{term2},  \eqref{norm1}, and comparing the sums with integrals (as done in the proof of Lemma \ref{expobound-mainterm}), we get that, for any $M>0$, 
\begin{equation} \label{caseBTermeRestBound-bis}
 {\mathbb P} \left(  \lambda W_p^p (\gm  , \delta_{x_0} )  \geq 3 x \right) 
      \leq   C_{p,r} \frac{\|X\|^{rp}_{rp, w}}{x^r n^{r-1}}
     + \frac { C_{p} }{x^2 n} \left(\int_0^M  \sqrt{H(t)} t^{p-1} dt    \right)^2     +  \frac {   C_p }{x  }  \int_ M^\infty   H(t)  t^{p-1} dt . 
\end{equation}
Since  $\mu$ satisfies a weak moment of order $rp$, then $H(t)  \leq t^{-rp} \|X\|^{rp}_{rp, w}$, and thus
\begin{eqnarray*} 
 {\mathbb P} \left(  \lambda W_p^p (\gm  ,\delta_{x_0} )  \geq 3 x \right)
    &\leq &  C_{p,r} \frac{\|X\|^{rp}_{rp, w}}{x^r n^{r-1}} + C_{p,r} \|X\|^{rp}_{rp, w} \left\{  \frac {M^{p(2-r)}}{x^2 n }  
     + \frac { 1 }{x M^{p(r-1)} } \right\} .
  \end{eqnarray*}
Taking $M =(nx)^{1/p}$, we obtain
\begin{eqnarray} 
\label{eq:Rem-geq-r-alpha}
 {\mathbb P} \left(  \lambda W_p^p (\gm  ,\delta_{x_0} )  \geq   3x \right)  
    \leq   C_{p,r}  \frac{\|X\|^{rp}_{rp, w}}{x^r n^{r-1}},
  \end{eqnarray}  
where  $C_{p,r}$ only depends on $p$ and $r$.

 It can be checked that the same inequality is valid for 
 ${\mathbb P} \left(  \lambda W_p^p (\gn  , \mu )  \geq 3 x \right)$. We conclude by combining Inequalities~\eqref{eq:MainTermcaseA-bis} and~\eqref{eq:Rem-geq-r-alpha}.
 
\smallskip  
 
\noindent $\bullet$ We now only assume that $\mu$ has a weak moment of order $rp$ with  $r \in (1,\alpha/(\alpha-p))$. We can find $q \in (r, 2) $
such that $\alpha < \frac{q}{q-1}p$. More precisely, we take $ q = 1/2 \left(r  + \frac{\alpha}{\alpha - p} \right)$ , and then $q$ appears as a function of $\alpha$, $p$ and $r$. To control the main term of   \eqref{borne3termes-bis}, we start by splitting the probability into three terms 
\begin{multline}\label{dec3xmain}
{\mathbb P} \left(\sum_{i \geq 1} \lambda_i W_p^p (L_n^i  , \mu^{K_i}  ) \geq 3x \right)   \leq {\mathbb P} \left(\lambda_1 W_p^p (L_n^1  , \mu^{K_1}  ) \geq  x \right) + {\mathbb P} \left(\sum_{i  =2 }^K \lambda_i W_p^p (L_n^i  , \mu^{K_i}  ) \geq  x \right) \\
+ {\mathbb P} \left(\sum_{i > K} \lambda_i W_p^p (L_n^i  , \mu^{K_i}  ) \geq  x  \right) \, .
\end{multline}
For the first term on the right-hand side in \eqref{dec3xmain}, we apply Markov's inequality at order $r$. We get
$$
{\mathbb P} \left(\lambda_1 W_p^p (L_n^1  , \mu^{K_1}  ) \geq  x \right) \leq \frac{\| \lambda_1 W_p^p (L_n^1  , \mu^{K_1}  ) \|_2^r}{x^r} \leq \frac{ 2^{rp}   C_{\alpha,p,\beta, r}    \mu(K_1)   ^{\frac{(\alpha-p)r}{\alpha}}}{x^r n ^{\frac{pr}{\alpha}}} \, ,
$$
where we used the bound \eqref{ineq:conditnLKi} with $q=2$ for the second inequality. Since we have assumed that $\|d(x_0,X)\|_p=1$ (implying that $\|X\|^{rp}_{rp, w} \geq C_{p,r}>0$), and since $pr/\alpha>r-1$, we obtain  that 
\begin{equation}\label{firstTermCaseE}
   {\mathbb P} \left(\lambda_1 W_p^p (L_n^1  , \mu^{K_1}  ) \geq  x \right) \leq  C_{\alpha, p, \beta, r}  \frac{\|X\|^{rp}_{rp, w}}{x^r n^{r-1}} \, .
\end{equation}
For the second term on the right-hand side in \eqref{dec3xmain}, we apply Markov's inequality at order $q$:
\begin{equation}
\label{eq:CaseEMainTerm-bis}
{\mathbb P} \left(\sum_{i  =2 }^K \lambda_i W_p^p (L_n^i  , \mu^{K_i}  ) \geq  x \right)
  \leq    \frac 1{x^q} \left(  \sum_{i = 2}^K  \| \lambda_i W_p^p (L_n^i  , \mu^{K_i}  ) \|_q \right)^q \leq \frac 1{x^q} \left(  \sum_{i = 2}^K  \| \lambda_i W_p^p (L_n^i  , \mu^{K_i}  ) \|_2 \right)^q \, .
\end{equation}
Proceeding as in \eqref{ineq:conditnLKi},  we  apply Proposition~\ref{prop:bound-r-bounded} (case $\alpha >2p$) to each term of the sum conditionally to the sample size $n L_n(K_i)$. Since $p/\alpha >(q-1)/q$, we obtain
\begin{eqnarray*}
 \| \lambda_i W_p^p (L_n^i  , \mu^{K_i}  ) \|_2 
&\leq&  C_{\alpha,p,\beta}
 \left \| \frac{\lambda_i  2^{ip}}{ \left( n L_n(K_i) \right)^{p/\alpha }} \right \|_2 \\
 &\leq&   C_{\alpha,p,\beta} \left \| \frac{  L_n(K_i) ^{(q-1)/q} \mu(K_i)^{1/q}  2^{ip}}{ \left( n L_n(K_i) \right)^{(q-1)/q}} \right \|_2 \\
  &\leq&   C_{\alpha,p,\beta} \frac{    \mu(K_i)^{1/q}  2^{ip}}{   n ^{(q-1)/q}} .
\end{eqnarray*}
Combining this upper upper bound with \eqref{eq:CaseEMainTerm-bis} (and bearing in mind that $q$ is a function of $\alpha, p, r$), we get
\begin{equation}
\label{eq:CaseEMainTerm-bis2}
{\mathbb P} \left(\sum_{i  =2 }^K \lambda_i W_p^p (L_n^i  , \mu^{K_i}  ) \geq  x \right)
   \leq \frac {C_{\alpha,p,\beta, r}}{x^q   n ^{q-1} } \left(  \sum_{i = 2}^K \mu(K_i)^{1/q}  2^{ip} \right)^q \, .
\end{equation}
For the third term on the right-hand side in \eqref{dec3xmain}, we apply Markov's inequality at order $1$:
\begin{equation}\label{thirdTermCaseE}
{\mathbb P} \left(\sum_{i > K} \lambda_i W_p^p (L_n^i  , \mu^{K_i}  ) \geq  x  \right) \leq 
 \frac 1{x}  \sum_{i >K}  \left \|  \lambda_i   W_p^p (L_n^i  , \mu^{K_i}  ) \right \|_1 \leq \frac 1{x}  \sum_{i >K} 2^{ip}\mu(K_i) \, .
\end{equation}
Gathering \eqref{dec3xmain}, \eqref{firstTermCaseE}, \eqref{eq:CaseEMainTerm-bis2} and \eqref{thirdTermCaseE}, and  comparing the sums with integrals (as done in the proof of Lemma \ref{expobound-mainterm}), we get that, for any $M>0$, 
\begin{eqnarray*}
{\mathbb P} \left(\sum_{i \geq 1} \lambda_i W_p^p (L_n^i  , \mu^{K_i}  ) \geq 3 x \right)   
& \leq &   C_{\alpha, p, \beta, r}  \frac{\|X\|^{rp}_{rp, w}}{x^r n^{r-1}} +    \frac {C_{p,\alpha,\beta, r}} {x^q n ^{q-1}} \left(\int_0^M  H(t)^{1/q} t^{p-1} dt    \right)^q  
     +   \frac { C_{p}}{x  }  \int_ M^\infty   H(t)  t^{p-1} dt  .
\end{eqnarray*}
Since  $\mu$ has a weak moment of order $rp$, then $H(t)  \leq t^{-rp} \|X\|^{rp}_{rp, w}$, and thus 
\begin{eqnarray*} 
 {\mathbb P} \left(\sum_{i \geq 1} \lambda_i W_p^p (L_n^i  , \mu^{K_i}  ) \geq 3 x \right)
    &\leq & C_{\alpha, p, \beta, r}  \frac{\|X\|^{rp}_{rp, w}}{x^r n^{r-1}} + C_{p,\alpha,\beta,r}  \|X\|^{rp}_{rp, w} \left\{  \frac {1}{x^q n^{q-1} }  M^{p(q-r)} 
     + \frac {   1 }{x M^{p(r-1)} }  \right\}.
  \end{eqnarray*}
Taking $M =(nx)^{1/p}$, we obtain  
\begin{equation}
\label{caseEMainTerm}
{\mathbb P} \left(\sum_{i \geq 1} \lambda_i W_p^p (L_n^i  , \mu^{K_i}  ) \geq 2 x \right) 
    \leq  C_{p,\alpha,\beta,r}   \frac{\|X\|^{rp}_{rp, w}}{x^r n^{r-1}}.
\end{equation}
Regarding the remainder terms of~\eqref{borne3termes-bis}, the proof given for the previous case when $\mu$ has a weak moment of order $rp$ with  $r \in (\alpha/(\alpha-p),2)$ is also valid for $r \in (1,\alpha/(\alpha-p))$, and thus the bound \eqref{eq:Rem-geq-r-alpha} is satisfied (and is also valid  for 
${\mathbb P} \left(  \lambda W_p^p (\gn  , \mu )  \geq 3 x \right)$). The desired inequality follows from \eqref{eq:Rem-geq-r-alpha} and \eqref{caseEMainTerm}.

\smallskip

\noindent $\bullet$ We now assume that $I_{2p,p}=\int_0^\infty  \sqrt{H(t)} t^{p-1} dt <\infty$. For the main term of \eqref{borne3termes-bis}, note that $\alpha> 2p$ implies that $\frac {\alpha-p}{\alpha}>\frac 1 2$. Hence, $I_{\alpha,p} < \infty $. Thanks to Lemma \ref{expobound-mainterm}, we thus have \begin{align*}
     {\mathbb P} \left(\sum_{i \geq 1} \lambda_i W_p^p (L_n^i  , \mu^{K_i}  ) \geq x \right)\leq
     e^2 \exp \left( - n C_{p,\alpha,\beta} 
\left[ \frac
  {x} 
{  I_{\alpha,p}   } \right]^{\frac\alpha p} \right){\bf 1}_{x \leq C_p}.
\end{align*}
Regarding the remainder terms of \eqref{borne3termes-bis}, we start from \eqref{Lemma1again}, and we  use Markov's inequality at order 2, the upper bound \eqref{normq1} with $q=2$, and a comparison of sums with integrals. We obtain that
\begin{align}
{\mathbb P} (\lambda W_p^p(\gm  ,\delta_{x_0} )>x)&\leq \frac{1}{x^2} \left(\sum_{i\geq 1}\|L_n(K_i)-\mu(K_i)\|_2 r_i^p\right)^2 \nonumber \\
        &\leq \frac{1}{x^2} \left(\sum_{i\geq1}\sqrt{\frac{\mu(K_i)}{n}}2^{ip}\right)^2
    \leq \frac{C_p}{x^2 n}\left(\int_0^\infty  \sqrt{H(t)} t^{p-1} dt\right)^2 \, , \label{remainder2}
\end{align} 
and the same inequality is true for $P(\lambda W_p^p(\gn  ,\delta_{x_0}  )>x)$. We conclude by combining these two bounds.

\smallskip  

\noindent $\bullet$ We finally assume that $I_{\alpha,p}=\int_0^\infty   H(t)^{(\alpha-p)/\alpha }  t^{p-1} dt <\infty$. Since $\alpha > 2p$, we have $\frac{\alpha}{\alpha-p}\in(1,2)$. To control the main term of   \eqref{borne3termes-bis}, we start by splitting the probability into two terms:
\begin{equation}\label{dec2xmain}
{\mathbb P} \left(\sum_{i \geq 1} \lambda_i W_p^p (L_n^i  , \mu^{K_i}  ) \geq 2x \right)   \leq {\mathbb P} \left(\lambda_1 W_p^p (L_n^1  , \mu^{K_1}  ) \geq  x \right) + {\mathbb P} \left(\sum_{i  =2 }^\infty \lambda_i W_p^p (L_n^i  , \mu^{K_i}  ) \geq  x \right)  \, .
\end{equation}
For the first term on the right-hand side in \eqref{dec2xmain}, applying Markov's inequality at order $\alpha/(\alpha-p)$ and proceeding as to get \eqref{firstTermCaseE},  we obtain 
\begin{equation}\label{firstTermCaseG}
   {\mathbb P} \left(\lambda_1 W_p^p (L_n^1  , \mu^{K_1}  ) \geq  x \right) \leq  C_{\alpha, p, \beta}  \frac{I_{\alpha,p}^{\frac{\alpha}{\alpha-p}}}{x^\frac{\alpha}{\alpha-p} n^{\frac{p}{\alpha-p}}}\, .
\end{equation}
For the main term of \eqref{borne3termes-bis}, we apply Markov's inequality at order $\alpha/(\alpha-p)$, and we proceed as in \eqref{eq:CaseEMainTerm-bis}. We obtain  
\begin{equation*} 
{\mathbb P} \left(\sum_{i \geq 2} \lambda_i W_p^p (L_n^i  , \mu^{K_i}  ) \geq x \right)  \leq \frac{1}{x^{\frac{\alpha}{\alpha-p}}} \left(\sum_{i\geq 2}\left \|\lambda_i W_p^p(L_n^i,\mu^{K_i})\right \|_2\right)^{\frac{\alpha}{\alpha-p}}.\end{equation*}
Proceeding as in \eqref{ineq:conditnLKi},  we  apply Proposition~\ref{prop:bound-r-bounded} (case $\alpha >2p$) to each term of the sum conditionally to the sample size $n L_n(K_i)$. We obtain
\begin{equation*}
    \|\lambda_i W_p^p(L_n^i,\mu^{K_i})\|_2
    \leq C_{p,\alpha,\beta} \left \|\frac{\lambda_i 2^{ip}}{(nL_n(K_i))^{p/\alpha}} \right \|_2 \leq C_{p,\alpha,\beta}\frac{2^{ip}\mu(K_i)^{\frac{\alpha-p}{\alpha}}}{n^{p/\alpha}}.
\end{equation*}
Combining these two upper bounds, and  comparing the sums with integrals (as done in the proof of Lemma \ref{expobound-mainterm}), we get that 
\begin{equation}\label{eq:CaseGMainTerm-bis}
{\mathbb P} \left(\sum_{i \geq 2} \lambda_i W_p^p (L_n^i  , \mu^{K_i}  ) \geq x \right)   \leq
C_{\alpha, p, \beta}  \frac{I_{\alpha,p}^{\frac{\alpha}{\alpha-p}}}{x^\frac{\alpha}{\alpha-p} n^{\frac{p}{\alpha-p}}}\, .
\end{equation}
For the remainder terms of \eqref{borne3termes-bis},  starting from \eqref{Lemma1again} we get that
\begin{equation}\label{pain2}
{\mathbb P} \left(  \lambda W_p^p (\gm  ,\delta_{x_0} )  \geq 2 x \right)  \leq  {\mathbb P} \left(   [(L_n(K_1) - \mu(K_1)) \vee 0  ] r_i^p    \geq   x \right) +
  {\mathbb P} \left( \sum_{i \geq 2}^K  [(L_n(K_i) - \mu(K_i)) \vee 0  ] r_i^p    \geq   x \right) \, .
\end{equation}
 For the first term on the right-hand side in \eqref{pain2}, we apply Markov's inequality at order $\alpha/(\alpha-p)$, and we proceed as in \eqref{term1}. This gives 
\begin{equation}\label{term1G}
{\mathbb P} \left(   [(L_n(K_1) - \mu(K_1)) \vee 0  ] r_i^p    \geq   x \right) \leq C_{\alpha, p}  \frac{I_{\alpha,p}^{\frac{\alpha}{\alpha-p}}}{x^\frac{\alpha}{\alpha-p} n^{\frac{p}{\alpha-p}}}\, . 
\end{equation}
For the second term on the right-hand side in \eqref{pain2}, we apply Markov's inequality at order $\alpha/(\alpha-p)$. This gives 
\begin{equation*}
{\mathbb P}\left( \sum_{i \geq 2}^K  [(L_n(K_i) - \mu(K_i)) \vee 0  ] r_i^p    \geq   x \right) \leq    \frac{1}{x^{\frac{\alpha}{\alpha-p}}} \left(\sum_{i\geq 1}\|L_n(K_i)-\mu(K_i)\|_{\frac{\alpha}{\alpha-p}} r_i^p\right)^{\frac{\alpha}{\alpha-p}} \, .
\end{equation*}
Applying the von Bahr-Esseen inequality as recalled in Theorem \ref{vBE} of the appendix, we get 
\begin{equation*}
    {\mathbb P} \left( \sum_{i \geq 2}^K  [(L_n(K_i) - \mu(K_i)) \vee 0  ] r_i^p    \geq   x \right)
        \leq  \frac{C_{ \alpha,p}}{x^{\frac{\alpha}{\alpha-p}}n^{\frac{p}{\alpha-p}}} \left(\sum_{i\geq 2} 2^{ip}\mu(K_i)^{\frac{\alpha-p}{\alpha}} \right)^{\frac{\alpha}{\alpha-p}} \, .
\end{equation*} 
Comparing the sums with integrals (as done in the proof of Lemma \ref{expobound-mainterm}),  we obtain
\begin{equation}\label{term2G}
    {\mathbb P} \left( \sum_{i \geq 2}^K  [(L_n(K_i) - \mu(K_i)) \vee 0  ] r_i^p    \geq   x \right) \leq C_{\alpha, p}  \frac{I_{\alpha,p}^{\frac{\alpha}{\alpha-p}}}{x^\frac{\alpha}{\alpha-p} n^{\frac{p}{\alpha-p}}} \, .
\end{equation}
From \eqref{pain2}, \eqref{term1G} and \eqref{term2G}, we obtain that 
\begin{equation}\label{OUF3!!!}
{\mathbb P} \left(  \lambda W_p^p (\gm  ,\delta_{x_0} )  \geq 2 x \right) \leq C_{\alpha, p}  \frac{I_{\alpha,p}^{\frac{\alpha}{\alpha-p}}}{x^\frac{\alpha}{\alpha-p} n^{\frac{p}{\alpha-p}}} \, ,
\end{equation}
and one can check that the same inequality holds for 
${\mathbb P}(\lambda W_p^p(\gn  ,\delta_{x_0}  )>x)$. 

The desired inequality follows by combining \eqref{dec2xmain}, \eqref{firstTermCaseG}, \eqref{eq:CaseGMainTerm-bis} and \eqref{OUF3!!!}.

\subsection{Proof of Theorem~\ref{theo:dev_unbounded-alphaleq2p}}

Recall that  $\alpha < 2p$, and that we assume that $\|d(x_0, X)\|_p=1$.

\smallskip

\noindent $\bullet$ We start with the case where $\mu$ has weak moments of order $rp$ with $r >  2$.   Regarding the main term of \eqref{borne3termes-bis}, since the moment condition on $\mu$ imply that  $I_{2p,p} = \int t^{p-1} \sqrt{H(t)} dt < \infty$,  Lemma~\ref{expobound-mainterm} gives that
\begin{equation}
 \label{eq:MainCaseB}   
{\mathbb P} \left(\sum_{i \geq 1} \lambda_i W_p^p (L_n^i  , \mu^{K_i}  ) \geq x \right)    \leq  e^2 \exp \left( - n C_{p,\alpha,\beta} \left[ \frac
  {x} { I_{2p,p}  } \right]^{2} \right) {\bf 1}_{x \leq C_p} \, .
\end{equation}
Regarding the remainder terms of \eqref{borne3termes-bis}, the computations we made in the proof of Theorem~\ref{theo:dev_unbounded-alphageq2p}, when $\mu$ has a weak moment of order $rp$ with $r >  2$, are still valid, and thus Inequality~\eqref{caseA:rest-bis}  is satisfied for any $q >r$.

We conclude by combining \eqref{eq:MainCaseB}   and \eqref{caseA:rest-bis}.

\smallskip 

\noindent $\bullet$ 
Assume that $\mu$ has a weak moment of order $rp$ with $r \in (1,2) $.  In order to control the main term of  \eqref{borne3termes-bis}, we proceed as in the proof of Theorem~\ref{theo:dev_unbounded-alphageq2p}  for the case
$r \in (1,\alpha/(\alpha-p))$: we start from \eqref{dec3xmain}, splitting the probability into three terms. For the first term on the right-hand side in \eqref{dec3xmain}, we apply Markov's inequality at order $r$. We get
$$
{\mathbb P} \left(\lambda_1 W_p^p (L_n^1  , \mu^{K_1}  ) \geq  x \right) \leq \frac{\| \lambda_1 W_p^p (L_n^1  , \mu^{K_1}  ) \|_2^r}{x^r} \leq \frac{ 2^{rp}   C_{\alpha,p,\beta, r}    \mu(K_1)   ^{\frac{r}{2}}}{x^r n ^{\frac{r}{2}}} \, ,
$$
where we used the bound \eqref{ineq:conditnLKi2} with $q=2$ for the second inequality. Since we have assumed that $\|d(x_0,X)\|_p=1$ (implying that $\|X\|^{rp}_{rp, w} \geq C_{p,r}>0$), and since $r/2>r-1$, we obtain  that Inequality~\eqref{firstTermCaseE} holds.
For the second term on  the right-hand side in \eqref{dec3xmain}, we apply Markov's inequality at order $2$:
\begin{equation*}
{\mathbb P} \left(\sum_{i  =2 }^K \lambda_i W_p^p (L_n^i  , \mu^{K_i}  ) \geq  x \right)
  \leq    \frac 1{x^2} \left(  \sum_{i = 2}^K  \| \lambda_i W_p^p (L_n^i  , \mu^{K_i}  ) \|_2 \right)^2  \, .
\end{equation*}
Combining this upper upper bound with \eqref{ineq:conditnLKi2}, we get
\begin{equation}
\label{eq:CaseEMainTerm-bis3}
{\mathbb P} \left(\sum_{i  =2 }^K \lambda_i W_p^p (L_n^i  , \mu^{K_i}  ) \geq  x \right)
   \leq \frac {C_{\alpha,p,\beta}}{x^2   n  } \left(  \sum_{i = 2}^K 2^{ip} \sqrt{\mu(K_i)} \right)^2 \, .
\end{equation}
For the third term on the right-hand side in \eqref{dec3xmain}, we see that Inequality~\eqref{thirdTermCaseE} is valid.

Gathering \eqref{dec3xmain}, \eqref{firstTermCaseE}, \eqref{eq:CaseEMainTerm-bis3} and \eqref{thirdTermCaseE}, and  comparing the sums with integrals (as done in the proof of Lemma \ref{expobound-mainterm}), we get that, for any $M>0$, 
\begin{eqnarray*}
{\mathbb P} \left(\sum_{i \geq 1} \lambda_i W_p^p (L_n^i  , \mu^{K_i}  ) \geq 3 x \right)   
& \leq &   C_{\alpha, p, \beta, r}  \frac{\|X\|^{rp}_{rp, w}}{x^r n^{r-1}} +    \frac {C_{p,\alpha,\beta}} {x^2 n} \left(\int_0^M  \sqrt{H(t)} t^{p-1} dt    \right)^2  
     +   \frac { C_{p}}{x  }  \int_ M^\infty   H(t)  t^{p-1} dt  .
\end{eqnarray*}
Since  $\mu$ has a weak moment of order $rp$, then $H(t)  \leq t^{-rp} \|X\|^{rp}_{rp, w}$, and thus 
\begin{eqnarray*} 
{\mathbb P} \left(\sum_{i \geq 1} \lambda_i W_p^p (L_n^i  , \mu^{K_i}  ) \geq 3 x \right)
    &\leq & C_{\alpha, p, \beta, r}  \frac{\|X\|^{rp}_{rp, w}}{x^r n^{r-1}} + C_{p,\alpha,\beta,r}  \|X\|^{rp}_{rp, w} \left\{  \frac {1}{x^2 n }  M^{p(2-r)} 
     + \frac {   1 }{x M^{p(r-1)} }  \right\}.
  \end{eqnarray*}
Taking $M =(nx)^{1/p}$, we see that Inequality~\eqref{caseEMainTerm} holds.

Regarding the remainder terms of \eqref{borne3termes-bis}, we proceed exactly as in the proof of  Theorem~\ref{theo:dev_unbounded-alphageq2p} for the case where  $\mu$ satisfies a weak moment of order $rp$, with  $r \in (\alpha/(\alpha-p),2)$. This leads to the upper bound \eqref{eq:Rem-geq-r-alpha}.

We conclude by combining \eqref{caseEMainTerm}   and \eqref{eq:Rem-geq-r-alpha}.

\smallskip

\noindent $\bullet$ Assume that $I_{2p,p}=\int_0^\infty  \sqrt{H(t)} t^{p-1} dt <\infty$. To address the main term in \eqref{borne3termes-bis}, we apply Inequality~\eqref{optq-CaseB} with $q=2$, and we obtain that
\begin{equation}\label{MainTerm2}
{\mathbb P} \left(\sum_{i \geq 1} \lambda_i W_p^p (L_n^i  , \mu^{K_i}  ) \geq x \right)   \leq \frac{C_{p,\alpha,\beta}}{x^2n} \left(\int_0^\infty    \sqrt{H(t)} t^{p-1} dt \right)^2.
\end{equation}
Regarding the remainder terms of \eqref{borne3termes-bis}, the computation we made in the proof of Theorem~\ref{theo:dev_unbounded-alphageq2p} when $I_{2p,p} <\infty$ is still valid, so that  Inequality~\eqref{remainder2} holds.
We conclude by combining \eqref{MainTerm2} and \eqref{remainder2}.


\subsection{Proof of Theorem~\ref{theo:dev_unbounded-alphaeq2p}}

Recall that  $\alpha = 2p$, and that we assume that $\|d(x_0, X)\|_p=1$.

\noindent $\bullet$ Assume that $\mu$ has a weak moment of order $rp$ with $r>2$. For the main term of \eqref{borne3termes-bis}, we apply Lemma \ref{expobound-mainterm} (case $\alpha=2p$). Regarding the remainder terms of \eqref{borne3termes-bis}, the computation we made in the proof of Theorem~\ref{theo:dev_unbounded-alphageq2p}, when $\mu$ has weak moments of order $pr$ with $r >  2$, is still valid, and thus Inequality~\eqref{caseA:rest-bis} is satisfied for any $q >r$. We conclude by combining Lemma \ref{expobound-mainterm} (case $\alpha=2p$) and \eqref{caseA:rest-bis}.

\smallskip

\noindent $\bullet$  Assume that $I_{2p,p}=\int_0^\infty  \sqrt{H(t)} t^{p-1} dt <\infty$. We apply again   Lemma \ref{expobound-mainterm} (case $\alpha=2p$) to control the main term of \eqref{borne3termes-bis}. Regarding the remainder terms of \eqref{borne3termes-bis}, the computation we made in the proof of Theorem~\ref{theo:dev_unbounded-alphageq2p} when $I_{2p,p} <\infty$ is still valid, so that  Inequality~\eqref{remainder2} holds. We conclude by combining Lemma~\ref{expobound-mainterm} (case $\alpha=2p$) and \eqref{remainder2}.

\smallskip 

\noindent $\bullet$  Assume that $\mu$ has a weak moment of order $rp$ with $r \in (1,2)$.  For the main term of \eqref{borne3termes-bis}, we follow the proof of  Theorem~\ref{theo:dev_unbounded-alphageq2p} in the case where $\mu$ has a weak moment of order $rp$ with  $r \in (1,\alpha/(\alpha-p))$. 
We 
start from \eqref{eq:CaseEMainTerm-bis}
with $q=1+r/2$. 
Proceeding as in \eqref{ineq:conditnLKi},  we  apply Proposition~\ref{prop:bound-r-bounded} (case $\alpha =2p$) to each term of the sum conditionally to the sample size $n L_n(K_i)$. It gives
\begin{eqnarray*}
 \left \| \lambda_i W_p^p (L_n^i  , \mu^{K_i}  ) \right \|_2
&\leq&  C_{p,\beta} \left \|
 \frac{\lambda_i  2^{ip}}{ \sqrt{ n L_n(K_i)}} \log \left ( e + \frac{n L_n(K_i)}{2} \right ) \right \|_2 \\
 &\leq&   C_{p,\beta, r} \left \|\frac{  L_n(K_i) ^{(q-1)/q} \mu(K_i)^{1/q}  2^{ip}}{ \left( n L_n(K_i) \right)^{(q-1)/q}}\right \|_2 \\
  &\leq&   C_{p,\beta, r} \frac{    \mu(K_i)^{1/q}  2^{ip}}{   n ^{(q-1)/q}} ,
\end{eqnarray*}
where we used the definition of $\lambda_i$ and the fact that $q/(q-1)<1/2$. The end of the proof is the same as in the proof of Theorem~\ref{theo:dev_unbounded-alphageq2p}, and thus the upper bound \eqref{caseEMainTerm} is satisfied. 

Regarding the remainder terms of~\eqref{borne3termes-bis}, we proceed exactly as in the proof of Theorem~\ref{theo:dev_unbounded-alphageq2p} in the case where $\mu$ has a weak moment of order $rp$ with  $r \in (\alpha/(\alpha-p),2)$: this gives the upper bound   \eqref{eq:Rem-geq-r-alpha}. We conclude by combining \eqref{caseEMainTerm} and \eqref{eq:Rem-geq-r-alpha}.

\subsection{Proof of Theorem~\ref{theo:dev_unbounded-expo}}

We start for the preliminary results of Section~\ref{subs-prelim}.
Without loss of generality, we assume that $\|d(x_0, X)\|_p=1$, and we start from Inequality~\eqref{borne3termes-bis}. Since \eqref{FGexp} holds, the main term in \eqref{borne3termes-bis} is bounded via Lemma 2. It remains to give an upper bound for the two terms ${\mathbb P} \left(  \lambda W_p^p (\gm  , \delta_{x_0} )  \geq x \right)$ and  ${\mathbb P} \left(  \lambda W_p^p (\gn  , \delta_{x_0}  )  \geq x \right)$. From the inequalities \eqref{Lemma1again}, it suffices to give an upper bound for the quantity
 \begin{equation}\label{Lemma13FG}
 {\mathbb P}\left ( \sum_{i \geq 1}  \left |L_n(K_i) - \mu(K_i)\right | 2^{ip} >x  \right ) \, .
 \end{equation}
 To control \eqref{Lemma13FG}, we apply Lemma 13 of \cite{fournier2015rate} (case of sub/super exponential moments). Note that Lemma 13 of \cite{fournier2015rate} is established for ${\mathbb R}^d$-valued random variables, but it is easy to see that it also applies to our more general situation by replacing the definition of the quantity ${\mathcal E}_{\kappa, \lambda}(\mu)$ on page 708 in \cite{fournier2015rate} by the expression \eqref{FGexp}.

\bigskip

\noindent {\bf Acknowledgments.} This work was supported by the French ANR project  GeoDSIC ANR-22-CE40-0007.

\bibliographystyle{abbrvnat}
\bibliography{ReferencesW}

\appendix

\section{Moment inequalities for partial sums}

In this appendix, we summarise the moment inequalities used in this article, in the exact form in which we will use them. We provide two references for each inequality: the first is the historical reference, and the second is the reference that allows us to obtain the inequality as we state it. 

\begin{theorem}[Rosenthal Inequality for i.i.d. random variables. \cite{rosenthal1970subspaces}, \cite{pinelis1994optimum}]
\label{Lem-Ros}
There exists an  absolute constant $L$ such that for any sequence of i.i.d and centered random variables
 $(\xi_i)_{1\leq i\leq n}$ with finite $r$-th moment with $r \geq 2$,  one has
$$ 
\left\| \frac 1n \sup_{1\leq k \leq n} \left | \sum_{i=1} ^k  \xi_i \right | \, \right\|_r \leq L  \left( \sqrt{\frac{r}{n}}   \|\xi_1\|_2   + \frac{r}{n} \left\| \max_{i=1 \dots n} |\xi_i| \right\|_r   \right)
 \leq L  \left( \sqrt{\frac{r}{n}}   \|\xi_1\|_2   + \frac{r}{n^{(r-1)/r}}\left\|  \xi_1\right\|_r   \right).
$$
\end{theorem}

\begin{theorem}[Burkholder Inequality for i.i.d random variables. \cite{Burkholder73}, \cite{Rio2009}] \label{Burkholder} For any sequence of i.i.d and centered random variables
 $(\xi_i)_{1\leq i\leq n}$ with finite $r$-th moment with $r  \geq 2$,  one has
$$
\left\| \frac 1n   \sum_{i=1} ^n  \xi_i  \, \right\|_r  \leq \sqrt{\frac{r-1}{n}} \|\xi_1\|_r \, .
$$
\end{theorem}

\begin{theorem}[von Bahr-Esseen Inequality for i.i.d random variables. \cite{BahrEsseen}, \cite{Pinelis2015}] \label{vBE} For any sequence of i.i.d and centered random variables
 $(\xi_i)_{1\leq i\leq n}$ with finite $r$-th moment with $r \in (1, 2]$,  one has
$$
\left\| \frac 1n   \sum_{i=1} ^n  \xi_i  \, \right\|_r  \leq \frac{2^{(2-r)/r}}{n^{(r-1)/r}} \|\xi_1\|_r \, .
$$
\end{theorem}
\end{document}